\documentclass[a4,11pt]{article}

\usepackage{amsmath,amsfonts,amsthm,latexsym,amssymb,mathrsfs}
\usepackage{bbm}
\usepackage{enumerate}
\usepackage{a4wide}
\usepackage{eucal}
\usepackage{cite}
\usepackage{bbm}

\newcommand{\uno}{\mathbbmss{1}}

\newcommand{\gap}[2]{\widehat{\delta} \left(#1,#2 \right)}

\def\H{\mathcal{H}} 
\def\K{\mathcal{K}} 
\def\M{\mathcal{M}}
\def\N{\mathcal{N}}
\def\U{\mathcal{U}}
\def\V{\mathcal{V}}
\def\R{\mathcal{R}}
\def\A{\mathcal{A}}
\def\X{\mathcal{X}}
\def\Y{\mathcal{Y}}
\def\L{\mathcal{L}}
\def\S{\mathcal{S}}
\def\T{\mathcal{T}}
\def\OIP{$A^{(2)}_{\T, \S}$ }

\newtheorem{df}{Definition}[section]
\newtheorem{thm}[df]{Theorem}
\newtheorem{pro}[df]{Proposition}
\newtheorem{cor}[df]{Corollary}

\newtheorem{rema}[df] {Remark}
\newtheorem{lem}[df] {Lemma}

\title{Further results on the $(b, c)$-inverse, the outer inverse $A^{(2)}_{\T, \S}$ and the Moore-Penrose inverse in the Banach context}

\author{Enrico Boasso}
\date{ }
\begin{document}
\maketitle

\begin{abstract}\noindent In this article properties of the $(b, c)$-inverse, the inverse along an element, the outer inverse with prescribed range and null space \OIP and the Moore-Penrose 
inverse will be studied in the contexts of Banach spaces operators, Banach algebras and $C^*$-algebras. The main properties to be considered are the continuity, the differentiability and the openness of the sets of all invertible elements defined by all the aforementioned outer inverses but the Moore-Penrose inverse. The relationship between the $(b, c)$-inverse and the outer inverse \OIP will be also characterized.\par
	\par
\medskip
\noindent {\bf AMS Subjects Classifications:} 46H05; 46L05; 47A05; 15A09
\end{abstract}
 
\section{Introduction}

Recently two outer inverses have been introduced: the $(b, c)$-inverse and the inverse along an element, see \cite{D} and \cite{mary}, respectively.
These two inverses are related; in fact, the latter is a particular case of the former. It is worth noticing one of the main properties of these inverses, namely, they encompass
several well known outer inverses such as the Drazin inverse, the group inverse and the Moore-Penrose inverse.
Furthermore, several authors have studied these notions, see for the $(b, c)$-inverse
\cite{D,D2,CCW,KC,WCGC,KWC,b2,r} and for the invese along an element \cite{mary, M2, marybis, mary2, bb1, bb2, ZCLG, ZPCZ}.
In particular, in \cite{b2} and  \cite{bb2} several properties of the $(b, c)$-inverse and the inverse along an element  were studied in the Banach context, respectively.

On the other hand, one of the most well known outer inverses is the outer inverse with prescribed range and null space, i.e., the outer invers \OIP.
This inverse has been studied in the frames of matrices, Hilbert space operators and Banach space operators. To learn about  the \OIP outer inverse in Banach spaces, see for example \cite{YW, DX, LYZW}.

The main objective of this article is to deepen the knowledge of the three aformentioned outer inverses in the contexts of Banach algebras, $C^*$-algebras, Banach space operators
and Hilbert space operators. However, as an application of the main results, properties of the Moore-Penrose inverse will be also presented.

The article is organized as follows. In section 3, after having recalled some preliminary definitions and facts in section 2, the relationship between the $(b, c)$-inverse (in particular the inverse along an element)
and the outer inverse \OIP will be studied. In section 4 both the set of all $(b, c)$-invertible elements (in particular the set of all invertible elements along a fixed element) 
and the set of all operator for which the outer inverse \OIP exists will be proved to be open. The continuity of the $(b, c)$-inverse of Banach space operators and of Banach algebra and $C^*$-algebra elements will be characterized in 
section 5; two main notions will be used to accomplish this aim: the gap between two subspaces and the Moore-Penrose inverse. The diffrentiability of the $(b, c)$-inverse in Banach algebras and $C^*$-algebras will be studied in section 6; the Moore-Penrose inverse will be also applied in this section. Finally, the continuity and differentiability of the outer inverse \OIP will be characterized in section 7 using again the gap between subspaces and the Moore-Penrose inverse. In addition, in section 5 and 6, as an application of the main results of these sections,
the continuity and the differentiability of the Moore-Penrose inverse for Banach algebra elements and Banach space operators will be studied, respectively.

\section{Preliminary definitions}

\noindent From now on, $\A$ will denote a unitary Banach algebra with unit $\uno$  while $\A^{-1}$ and $\A^\bullet$ will stand for the set 
of invertible elements and the set of idempotents of $\A$, respectively. A particular case is  $\L (\X)$, the Banach algebra of all
linear and bounded maps defined on and with values in the Banach space $\X$.  However, in the present work it will be necessary to consider the Banach space
of all operators defined on the Banach space $\X$ with values in the Banach space $\Y$, which will be denoted by $\L(\X, \Y)$. Note that if $T\in\L (\X, \Y)$, then $\N(T)\subseteq\X$ and $\R(T)\subseteq \Y$ will stand for 
the null space and the range of the operator $T$, respectively. For example, when  $\A$ is a unitary Banach algebra and $x\in \A$,  the operators $L_x\colon \A\to \A$ and $R_x\colon \A\to \A$ 
are the maps defined as follows: given $z\in \A$, $L_x(z)= xz$ and $R_x (z)= zx$. Observe that since $\A$ is unitary, then $\parallel L_a\parallel=\parallel a\parallel=\parallel R_a\parallel$. Moreover, the following notation will be used:
\begin{align*}
&x^{-1}(0)=\N(L_x),&  &x\A=\R(L_x),& &x_{-1}(0)=\N(R_x),& &\A x=\R(R_x).&
\end{align*}
Note that  when no confusion is possible, the identity operator defined on the 
Banach space $\X$ will be denoted by $I\in\L(\X)$; otherwise it will be denoted by $I_{\X}$. In addition, given  a Hilbert space $\H$ and  a closed subspace $\M\subseteq \H$, $P_{\M}^\perp\in\L(\H)^\bullet$ wil stand for the  orthogonal projector with range $\M$. \par

An element $a \in \A$ will be said to be {\it regular}, if there exists $x \in \A$ such that $a=axa$.
The element $x$, which is not uniquely determined by $a$, will be said to be  a {\it generalized inverse}
or an {\it inner inverse} of $a$. In addition, $\hat{\A}$ will stand for the set of all regular elements of $\A$ and given $a\in\hat{\A}$, $a\{1\}$ will denote the set of all
generalized inverses of $a$. On the other hand,  
if $y \in \A$ satisfies $yay=y$, then $y$  will be said to be an {\it outer inverse} of $a$. Moreover, an element $z$ will be said to be a \it normalized 
generalized inverse of $a$,  \rm if $z$ is both an inner and an outer inverse of $a$. Recall that if $b$ is an inner inverse of $a$, then
$bab$ is a normalized generalized inverse of $a$.

Now the definition of one of the key notion of this article will be recalled. Note that this notion was originally introduced in the context of semigroup, however,
since the frame of this article are Banach algebras and Banach space operators, the notion under consideration, as well as all the object considered in this work, will be introduced and studied 
in the Banach context.\par

\begin{df}[{\hspace{-1pt}\cite[Definition 1.3]{D}}]\label{def1}
Let $\A$ be a unitary Banach algebra and consider $b, c \in \A$. The element $a\in\A$
will be said to be $(b,c)$-invertible, if there exists $y \in \A$ such that the following
equations hold:
\begin{enumerate}[{\rm (i)}]
\item$y\in(b\A y)\cap (y\A c)$,
\item $b=yab$, $c=cay$.
\end{enumerate}
\end{df}

\indent In the same conditions of Definition \ref{def1}, if such an inverse exists, then it is unique (\hspace{-1pt}\cite[Theorem 2.1 (i)]{D}). Thus in what follows,
if the element $y$ in Definition \ref{def1} exists, then it will be denoted by $a^{-(b,\hbox{ }c)}$. In addition,
$a^{-(b,\hbox{ }c)}$ is an outer inverse of $a$ (\hspace{-1pt}\cite[Theorem 2.1 (ii)]{D}) and $b$ and $c$ are regular (\hspace{-1pt}\cite[Remark 2.2 (iii)]{b2} or \cite[Proposition 3.3]{WCGC}). 

\indent A particular case of the $(b,c )$-inverse is the Bott-Duffin $(p, q)$-inverse.

\begin{df}[{\hspace{-1pt}\cite[Definition 3.2]{D}}]\label{def15}Let $\A$ be a unitary Banaxh algebra and consider $p$, $q\in\A^\bullet$. The element
$a\in\A$ will be said to be Bott-Duffin $(p, q)$-invertible, if there exists $y\in \A$ such that 
\begin{enumerate}[{\rm (i)}]
\item $y=py=yq$,
\item $yap=p$ and $qay=q$.
\end{enumerate}
\end{df}

\indent Clearly, given $p$, $q\in\A^\bullet$, the Bott-Duffin $(p, q)$-inverse is nothing but the $(b, c)$-inverse 
when $b$ and $c$ are idempotents. In addition, since there exists at most one $(b, c)$-inverse,
the Bott-Duffin $(p, q)$-inverse is unique, if it exists. According to what has been said, if $a\in\A$ is
Bott-Duffin  $(p,q)$-invertible, then the element $y$ in Definition \ref{def15} will be denoted by
$a^{-(p, \hbox{ }q)}$. To learn more on the outer inverses recalled in Definition \ref{def1} and Definition \ref{def15}, see \cite{D, D2, CCW, KC, WCGC, KWC, b2, r}.

Next the definition of the inverse along an element will be recalled. This is another particular case of the $(b, c)$-inverse.
 
\begin{df}[{\hspace{-1pt}\cite[Definition 4]{mary}}] \label{def2}Consider $a$, $d\in\A$. The element $a$ 
is invertible along $d$, if there exists $b\in\A$ such that
 $b$ is an outer inverse of $a$, $b\A=d\A$ and $\A b=\A d$.\end{df}

\indent Recall that, in the same conditions of Definition \ref{def2}, according to \cite[Theorem 6]{mary}, if such $b\in\A$ exists, then it is unique.
Therefore, the element $b$ satisfying  Definition \ref{def2} will be said to be the \it inverse of $a$ along $d$. \rm
In this case, the inverse under consideration will be denoted by $a^{-d}$. Moreover, according to \cite[Proposition 6.1]{D}, the inverse along an element
is a particular case of the $(b, c)$-inverse. In fact, the element $a$ is invertible along $d$ if and only if it is $(d, d)$-invertible.
Furthermore, in this case, $a^{-d}=a^{-(d, d)}$. To learn more on this outer inverse, see  \cite{mary, M2, marybis, mary2, bb1, bb2, ZCLG, ZPCZ}.

\indent One of the most studied generalized inverses is the outer inverse  \OIP, i.e., the outer inverse of the operator $A\in \L(\X, \Y)$ with 
prescribed range $\T\subseteq \X$ and null space $\S\subseteq Y$, where $\X$ and $\Y$ are Banach spaces. 
This generalized inverse was studied for matrices and for operators defined on Hilbert and on Banach spaces.  Recall that this inverse
is unique, when it exists (\hspace{-1pt}\cite[Lemma 1]{LYZW}).\par  

\begin{df}\label{def30}Let $\X$ and $\Y$ be Banach spaces, $A\in\L (\X, \Y)$ and $\T$ and $\S$ two closed subspaces of $\X$ and $\Y$, respectively. 
If there exists a {\rm(}necessarily unique{\rm)} operator $B\in\L (\Y, \X)$ such that
$B$ is an outer inverse of $A$, $\R(B)=\T$ and $\N(B)=\S$, then $B$ will be said to be the \OIP outer inverse of $A$. 
\end{df}
\indent According to \cite[Lemma 1]{LYZW}, a necessary and sufficient condition for the existence of \OIP  is that
$\T$ and $\S$ are complemented subspaces of $\X$ and $\Y$, respectively, $A\mid_\T^{A(\T )}\colon \T\to A(\T)$ is invertible and $A(\T)\oplus \S=\Y$. In particular, using this latter decomposition,
\OIP is the following operator: 
\begin{align*}
&A^{(2)}_{\T,\S}\mid_\S=0,&  &(A\mid_T^{A(\T)})^{-1}=A^{(2)}_{\T,\hbox{ }\S}\mid^T_{A(\T)}\colon A(\T)\to \T.&\\
\end{align*}
\noindent To learn more properties of the \OIP outer inverse in Banach spaces, see \cite{YW, DX, LYZW}.

\indent To characterize the continuity of the outer inverses considered in this article, it is necessary to recall the definition of the Moore-Penrose inverse. 
Let $\A$ be a $C^*$-algebra. An element $a\in\A$ will be said to be \it Moore-Penrose invertible, \rm if there exists  $b\in \A$ such that
the following equations hold:
\begin{align*}
&a=aba,&  &b=bab,& &(ab)^*=ab,& &(ba)^*=ba.&\\
\end{align*}

To give the notion of a Moore-Penrose invertible Banach algebra element, the definition of an hermitian element need to be recalled first.
Given a Banach algebra $\A$, an element $z\in\A$ is said to be \it hermitian, \rm if $\parallel exp(ita)\parallel =1$,  for all $ t\in\mathbb{R}$
(\hspace{-1pt}see \cite{Vv}, \cite[Chapter 4]{Dw} and \cite[Chapter I, Section 10]{BD}). When $\A$ is a $C^*$-algebra, $a\in\A$ is hermitian
if and only if it is self-adjoint (\hspace{-1pt}\cite[Proposition 20, Chapter I, Section 12]{BD}). Now Moore-Penrose Banach algebra elements can be defined.

Given a Banach algebra  $\A$, an element $x\in\A$ will be said to be \it Moore-Penrose invertible, \rm if 
there exists $b\in\A$ such that $b$ is a normalized generalized inverse of $a$ and the elements $ab$ and $ba$ are hermitian.
Since the Moore-Penrose inverse is unique, when it exists (see \cite[Lemma 2.1]{V3}), the element $b$ will be denoted by $a^\dag$. Moreover, $\A^\dag$ will stand for the set of all 
Moore-Penrose invertible elements of $\A$. Note that $\A^\dag\subseteq \hat{\A}$. Furthermore, when $\A$ is a $C^*$-algebra, $\A^\dag=\hat{\A}$
(\hspace{-1pt}\cite[Theorem 6]{HM1}); however, when the Banach algebra $\A$ is not a $C^*$-algebra, in general, this result does not hold (\hspace{-1pt}\cite[Remark 4]{B}).
To learn the definition of the Moore-Penrose inverse for matrices, see \cite{Pe}; to learn more
properties of this generalized inverse see, for $C^*$ algebras, \cite{HM1, HM2, kol}, and for Banach algebras, \cite{V2, V3, B}.

\indent To prove some of the main results of this article, the definition of the gap between two subspaces need to be recalled. 
Let $\X$ be a Banach space and consider $\M$ and $\N$ two closed subspaces in $\X$. If $\M=0$, then set $\delta (\M, \N)=0$,
otherwise set
$$
\delta (\M, \N)=\hbox{\rm sup}\{\hbox{\rm dist}(x,\N)\colon x\in\M, \, \parallel x\parallel =1\},
$$  
where $\hbox{\rm dist}(x,\N)=\hbox{\rm inf}\{\parallel x-y\parallel\colon y\in\N\}$. The \it gap between the subspaces $\M$ and $\N$ \rm  is 
$$
\gap{\M}{\N}= \hbox{\rm max}\{ \delta (\M, \N), \delta (\N, \M)\}. 
$$
\noindent To learn more on this notion, see \cite{DX, K}.

\section{The relationship between the $(b, c)$-inverse and \OIP}

\noindent The main objective of this section is to prove that given a Banach space $\X$, the $(B, C)$-inverse in $\L(\X)$ consists in a reformulation of 
the outer inverse \OIP, in other words, the outer inverse introduced in \cite[Definition 1.3]{D} is an extension to semigroups of the outer inverse
with prescribed range and null space. 

First an equivalent formulation of Definition \ref{def1} will be given in the context of Banach space operators. To this end, however, some
preparation is needed. In the following remark a particular operator will be introduced.

\begin{rema}\label{rem3.4}\rm
Recall that given a Banach space $\X$, a necessary and sufficient condition for $F\in \L(\X)$ to be a regular operator is that $\N(F)$ and $\R(F)$ are
closed and complemented susbpsaces of $\X$. Suppose that $F\in \L(\X)$ satisfies this condition and let  $\N$ and $\M$ be two  subspaces of $\X$ such that
$$
\N(F)\oplus \N=\X=\R(F)\oplus \M.
$$ 
Consider the isomorphism $F_1=F\mid_{\N}^{\R(F)}\colon \N\to \R(F)$ and define the map $S_F\in\L(X)$ as follows: 
\begin{align*}
& &&S_F\mid{\M}=0,& &S_F\mid_{\R(F)}=\iota_{\N, \X}F_1^{-1},&
\end{align*}
where $\iota_{\N, \X}\colon \N\to\X$ is the inclusion map. 
The operator $S_F\in\L(\X)$ will be used in the proofs of the next results.
\end{rema}

\begin{lem}\label{lem3.3} Let $\X$ be a Banach space and consider $F\in\L(\X)$ a regular operator and  $S_F\in\L(\X)$ the map defined in 
Remark \ref{rem3.4}. The following statements hold.\par
\begin{enumerate}[{\rm (i)}] 
\item If $X\in\L(\X)$ is such that $\R(X)=\R(F)$, then  $X=FS_FX$.
\item If $X\in\L(\X)$ is such that $\N(X)=\N(F)$, then $X=XS_FF$.
\end{enumerate}
\end{lem}
\begin{proof} Note that $FS_F\mid_{\R(F)}=\iota_{\R(F), \X}$, where $\iota_{\R(F), \X}\colon \R(F)\to \X$ is the inclusion map.
If $\R(X)=\R(F)$, then $X=FS_FX$.\par

Similarly, $S_FF\mid_{\N}=\iota_{\N, \X}$, where $\N$ is the subspace of $\X$ considered in Remak \ref{rem3.4}. 
If $\N(X)=\N(F)$, then since $\N(F)\oplus \N=\X$, $X=XS_FF$.
\end{proof}

The following proposition is a key step in the reformulation of Definition \ref{def1} for Banach space operators.

\begin{pro}\label{pro3.5} Let $\X$ be a Banach space and consider two regular operators $F$, $G\in\L(\X)$.
\begin{enumerate}[{\rm (i)}] 
\item A necessary and sufficient condition for $F\L(\X)=G\L(\X)$ is that $\R(F)=\R(G)$.
\item $\L(\X)F=\L(\X)G$ if and only if $\N(F)=\N(G)$. 
\end{enumerate}
\end{pro}
\begin{proof}If there exist $U$, $V\in\L(\X)$ such that $F=GU$ and $G=FV$, then $\R(F)=\R(G)$.
On the other hand, if $\R(F)=\R(G)$, then according to Lemma \ref{lem3.3} (i), $G=FS_FG$ and $F=GS_GF$,
which implies that $F\L(\X)=G\L(\X)$.

A similar argument, using in particular Lemma \ref{lem3.3} (ii), proves the second statement.
\end{proof}

\begin{thm}\label{thm3.1}Let $\X$ be a Banach space and consider two regular operators $B$, $C\in\L(\X)$.
The following statements are equivalent.
\begin{enumerate}[{\rm (i)}] 
\item The operator $A\in\L(\X)$ is $(B, C)$-invertible.
\item There exists a bounded and linear map $X\in\L(\X)$ such that 
$$
B=XAB, \hskip.3truecmC=CAX, \hskip.3truecm\R(X)=B,\hskip.3truecm \N(X)=\N(C).
$$
\end{enumerate}
Moreover, in this case $X=A^{-(B, C)}$.
\end{thm}
\begin{proof}
According to Definition \ref{def1}, if the $(B, C)$-inverse of $A$ exists, then there are $X$, $T_1$ and $T_2\in\L(\X)$ such that 
$$
B=XAB,  \hskip.3truecm C=CAX,  \hskip.3truecm X= BT_1, \hskip.3truecm X=T_2C.
$$
In particular, $\R(B)=\R(X)$ and $\N(X)=\N(C)$.\par

\indent On the other hand, suppose that statement (ii) holds. Since $B$ (respectively $C$) is regular and $\R(X)=\R(B)$ (respectively $\N(X)=\N(C)$), according to Lemma \ref{lem3.3},
$X=BS_BX$ (respectively $X=XS_CC$), where $S_B$ (respectively $S_C$) is the operator considered in Remark \ref{rem3.4}. Therefore, $X\in B\L(\X)X\cap X\L(\X)C$. Since the $(B, C)$-inverse is 
unique, when it exists, $X=A^{-(B, C)}$.
\end{proof}

Next the main result of this section will be proved

\begin{thm}\label{thm3.2} Let $\X$ be a Banach space and consider two regular operators $B$, $C\in\L(\X)$.
The following statements are equivalent.
\begin{enumerate}[{\rm (i)}] 
\item The operator $A\in\L(\X)$ is $(B, C)$-invertible.
\item There exists $X\in \L(\X)$ such that $X=XAX$, $\R(X)=\R(B)$ and $\N(X)=\N(C)$.
\item The outer inverse $A^{(2)}_{\R(B), \N(C)}$ exists.
\item $A\mid_{\R(B)}^{\R(AB)}\colon \R(B)\to\R(AB)$ is invertible and $\R(AB)\oplus \N(C)=\X$.
\end{enumerate}
\noindent Furthermore, in this case $A^{-(B, C)}=X=A^{(2)}_{\R(B),\, \\N(C)}$.
\end{thm}
\begin{proof}According to \cite[Proposition 6.1]{D}, statement (i) is equivalent to the fact that there exists an operator $X\in\L(\X)$
such that $X$ is an outer inverse of $A$, $X\L(\X)=B\L(\X)$ and $\L(\X)X=\L(\X)C$. However, according to the proof of Theorem \ref{thm3.1},
these conditions are equivalent to statement (ii). In addition, since the $(B,C)$-inverse is unique, when it exists, $X=A^{-(B, C)}$.

Statement (ii) and (iii) are equivalent; moreover $X=A^{(2)}_{\R(B),\, \N(C)}$ (\hspace{-1pt}\cite[Lemma 1]{LYZW}).

To prove the equivalence between statements (iii) and (iv), apply \cite[Lemma 1]{LYZW} and recall that, since $B$ and $C$ are regular, $\R(B)$ and $\N(C)$ are
closed and complemented subspaces of $\X$.
\end{proof}

Note that Theorem \ref{thm3.2} was proved for square matrices in \cite[Theorem 1.5]{r}. The following results will be derived from what has been proved.

\begin{cor}\label{cor3.6} Let $\X$ be a Banach space and consider $A$, $B$, $C\in\L(\X)$ such that $B$ and $C$ are regular. The following statements are equivalent.
\begin{enumerate}[{\rm (i)}] 
\item $A^{-(B, C)}$ exists.
\item $A^{-(F, G)}$ exists for any $F$, $G\in\L(\X)$, $F$,$G$ regular, such that $\R(F)=\R(B)$ and $\N(G)=\N(C)$.
\item The Bott-Duffin inverse $A^{-(P, Q)}$ exists for any $P$, $Q\in \L(\X)^\bullet$ such that $\R(P)=\R(B)$ and $\N(Q)=\N(C)$. 
\end{enumerate}
\noindent Furthermore, in this case, $A^{-(B, C)}=A^{-(F, G)}=A^{-(P, Q)}$.
\end{cor}
\begin{proof} Apply Theorem \ref{thm3.2}.
\end{proof}

Note that Corollary \ref{cor3.6} can be rephrased using the outer inverse $A^{(2)}_{\T, \S}$.

\begin{rema}\label{rema3.8}\rm Let $\X$ be a Banach space and consider $\S$ and $\T$ two closed and complemented subspaces of $\X$. Let $A\in \L(\X)$.
The following statement are equivalent.\par
\begin{enumerate}[{\rm (i)}] 
\item The outer inverse $A^{(2)}_{\T, \S}$ exists.
\item $A^{-(B, C)}$ exists for any $B$, $C\in\L(\X)$, $B$, $C$ regular, such that $\R(B)=\T$ and $\N(C)=\S$.
\item The Bott-Duffin inverse $A^{-(P, Q)}$ exists for any $P$, $Q\in \L(\X)^\bullet$ such that $\R(P)=\T$ and $\N(Q)=\S$. 
\end{enumerate}

\noindent Moreover, in this case $A^{(2)}_{\T, \S}=A^{-(B, C)}=A^{-(P, Q)}$.\par
\noindent The proof of the equivalence among statements (i)-(iii) and the last identity can be derived from Theorem \ref{thm3.2} and Corollary \ref{cor3.6}.
\end{rema}

It is worth noticing, as it has been mentioned in the first paragraph of this section,  that Theorem \ref{thm3.2} and Remark \ref{rema3.8} show that in $\L(\X)$, $\X$  a Banach space, the $(B, C)$-inverse is a reformulation of \OIP,
so that  \cite[Definition 1.3]{D} consists in an extension of the latter outer inverse to semigroups. 
In fact, given  $A\in\L(\X)$, \OIP and the $(B, C)$-inverse of $A\in\L(\X)$ refer to the same object -under the conditions of Theorem \ref{thm3.2} and Remark \ref{rema3.8}, but it could be said that the former outer inverse is defined
from a spatial point of view while the latter from an algebraic point of view. Actually, in the first case subspaces of an underlying space (a Banach space $\X$ or $\mathbb{C}^n$,
$n\in\mathbb{N}$) are used in the definition, but in the second only elements of a semigroup can be considered; precisely, in $\L(\X)$ it is possible to associate two specific subspaces to any element of this algebra -the range and the null space,
which leads to the aforementioned results.

Next the inverse along an operator will be considered.

\begin{cor}\label{cor3.9} Let $\X$ be a Banach space and consider $A$, $D\in\L(\X)$ such that $D$ is regular. The following statements are equivalent.
\begin{enumerate}[{\rm (i)}] 
\item $A^{-D}$ exists.
\item There exists $X\in \L(\X)$ such that $XAD=D=DAX$, $\R(X)=\R(D)$ and $\N(X)=\N(D)$.
\item There exists $Y\in\L(\X)$ such that $Y$ is an outer inverse of $A$, $\R(Y)=\R(D)$ and $\N(Y)=\N(D)$.
\item The outer inverse $A^{(2)}_{\R(D), \N(D)}$ exists.
\item $A\mid_{\R(D)}^{\R(AD)}\colon \R(D)\to\R(AD)$ is invertible and $\R(AD)\oplus \N(D)=\X$.
\item $A^{-F}$ exists for all $F\in\L(\X)$ such that  $F$ is regular,  $\R(F)=\R(D)$ and $\N(F)=\N(D)$.
\end{enumerate}
Moreover, in this case, $A^{-D}=X=Y=A^{(2)}_{\R(D), \N(D)}=A^{-F}$.
\end{cor}
\begin{proof} Recall that according to \cite[Proposition 6.1]{D}, $A^{-D}=A^{-(D, D)}$, when one of these outer inverses exists. 
Apply now Theorem \ref{thm3.1}, Theorem \ref{thm3.2} and Corollary \ref{cor3.6}.
\end{proof}

In the following proposition the Banach algebra case will be considered.

\begin{pro}\label{pro3.10} Let $\A$ be a unitary Banach algebra and consider $a$, $b$, $c\in\A$ such that $b$ and $c$ are regular.
The following statements hold.
\begin{enumerate}[{\rm (i)}] 
\item If $a^{-(b, c)}$ exists, then $L_{a^{-(b, c)}}=L_a^{-(L_b, L_c)}=(L_a)^{(2)}_{b\A, c^{-1}(0)}\in\L(\A)$.
\item Suppose that $L_a^{-(L_b, L_c)}=(L_a)^{(2)}_{b\A, c^{-1}(0)}\in \L(\A)$ exists and  there is $z\in\A$ such that $(L_a)^{(2)}_{b\A, c^{-1}(0)}=L_z$. 
Then, $a^{-(b, c)}$ exists and $a^{-(b, c)}=z$.
\item If $a^{-(b, c)}$ exists, then $R_{a^{-(b, c)}}=R_a^{-(R_c, R_b)}=(R_a)^{(2)}_{\A c, b_{-1}(0)}\in\L(\A)$.
\item Suppose that $R_a^{-(R_c, R_b)}=(R_a)^{(2)}_{\A c, b_{-1}(0)}\in \L(\A)$ exists and there is $w\in\A$ such that $(R_a)^{(2)}_{\A c, b_{-1}(0)}=R_w$. 
Then, $a^{-(b, c)}$ exists and $a^{-(b, c)}=w$.
\end{enumerate}
\end{pro}
\begin{proof} Recall that according to \cite[Proposition 6.1]{D}, necessary and sufficient for $a^{-(b, c)}$ to exist is that $a$ has an outer inverse, say $y$
($y=a^{-(b, c)}$), such that $y\A=b\A$ and $\A y=\A c$. Thus, if $a^{-(b, c)}$ exists, then $L_{a^{-(b, c)}}\in \L(\A)$ is an outer inverse of $L_a\in \L(\A)$ such
that 
\begin{align*}
\R(L_{a^{-(b, c)}})&=a^{-(b, c)}\A=b\A=\R(L_b),\\
\N(L_{a^{-(b, c)}})&=(a^{-(b, c)})^{-1}(0).
\end{align*}
\noindent However, note that $(a^{-(b, c)})^{-1}(0)=c^{-1}(0)=\N(L_c)$. In fact, since $\A a^{-(b, c)}=\A c$, there exist $u_1$, $v_1\in\A$ such that $a^{-(b, c)}=u_1c$ and $c=v_1a^{-(b, c)}$,
which implies that $(a^{-(b, c)})^{-1}(0)=c^{-1}(0)$. Therefore, according to \cite[Lemma 1]{LYZW} and Theorem \ref{thm3.2}, 
$$
L_{a^{-(b, c)}}=(L_a)^{(2)}_{b\A, c^{-1}(0)}=L_a^{-(L_b, L_c)}.
$$

Now suppose that statement (ii) holds. Note that according to Theorem \ref{thm3.2}, $L_a^{-(L_b, L_c)}\in \L(\A)$ exists if and only if 
$(L_a)^{(2)}_{b\A, c^{-1}(0)}\in \L(\A)$ exists; moreover, in this case, both maps coincide. 

Since $(L_a)^{(2)}_{b\A, c^{-1}(0)}\in \L(\A)$ is an outer inverse of $L_a\in\L(\A)$,  $z$ is an outer inverse of $a$. In addition, 
\begin{align*}
z\A &=\R(L_z)=\R((L_a)^{(2)}_{b\A, c^{-1}(0)})=b\A, \\
z^{-1}(0)&=\N(L_z)=\N((L_a)^{(2)}_{b\A, c^{-1}(0)})=c^{-1}(0).\\
\end{align*}
However, since $c$ and $z$ are regular, according to \cite[Proposition 3.1 (b) (iii)]{bb1}, $\A z=\A c$.  
Consequently, according to \cite[Proposition 6.1]{D}, $a^{-(b, c)}$ exists and $a^{-(b, c)}=z$.

As in the proof of statement (i), if $a^{-(b, c)}$ exists, then $R_{a^{-(b, c)}}$ is an outer inverse of $R_a\in \L(\A)$ such that
\begin{align*}
\R(R_{a^{-(b, c)}})&=\A a^{-(b, c)}=\A c=\R(R_c),\\
\N(R_{a^{-(b, c)}})&=(a^{-(b, c)})_{-1}(0).\\
\end{align*}
\noindent However, since $b\A=a^{-(b, c)}\A$, there exist $u_2$, $v_2\in\A$ such that $b=a^{-(b, c)}u_2$ and $a^{-(b, c)}=bv_2$, which implies that 
$(a^{-(b, c)})_{-1}(0)=b_{-1}(0)=\N (R_b)$. Therefore, according to \cite[Lemma 1]{LYZW} and Theorem \ref{thm3.2}, statement (iii) holds.

If statement (iv) holds, then as in the proof of statement (ii), observe that according to Theorem \ref{thm3.2}, $R_a^{-(R_c, R_b)}\in \L(\A)$ exists if and only if 
$(R_a)^{(2)}_{\A c, b_{-1}(0)}$ exists; moreover, in this case, both maps coincide. Furthermore, since $(R_a)^{(2)}_{\A c, b_{-1}(0)}\in \L(\A)$ is an outer inverse of $R_a\in\L(\A)$,  
$w$ is an outer inverse of $a$. In addition, 
\begin{align*}
\A w &=\R(R_w)=\R((R_a)^{(2)}_{\A c, b _1{-1}(0)})=\A c, \\
w_{-1}(0)&=\N(R_w)=\N((R_a)^{(2)}_{\A c, b_{-1}(0)})=b_{-1}(0).\\
\end{align*}
\noindent However, since $b$ and $w$ are regular, according to \cite[Proposition 3.1 (b) (iv)]{bb1}, $w\A=b\A$. 
Hence,  according to \cite[Proposition 6.1]{D}, $a^{-(b, c)}$ exists and $a^{-(b, c)}=w$.
\end{proof}

\section{Openness}

\noindent Given an unitary Banach algebra $\A$, $\A^{-1}\subset\A$ is an open set. In this section the set of all $(b ,c)$-invertible elements ($b$, $c\in\A$) will be proved to be open.
Similar results will be presented for the inverse along an element and the outer inverse \OIP.

Let $\A$ be a unitary Banach algebra and consider  $b$, $c\in \A$. Let $\A^{-(b, c)}$ be the set of all $(b, c)$-invertible elements of $\A$,
i.e., 
$$
\A^{-(b, c)}=\{a\in\A\colon a^{-(b, c)}\hbox{ exists}\}.
$$
Recall that if $b$ or $c$ is not regular, then $\A^{-(b, c)}=\emptyset$  (\hspace{-1pt}\cite[Remark 2.2 (iii)]{b2}). In addition, note that if $b=0=c$,
then $\A^{-(b, c)}=\A$. In fact, in this case, given $a\in\A$,  $a^{-(b, c)}=0$ satisfies Definition \ref{def1}. Moreover, if $b=0$ and $c\in\hat{\A}$ is such that $c\neq 0$ (respectively if $b\in\hat{\A}$ is such that 
$b\neq 0$ and $c=0$),
according to Definition \ref{def1}, then $\A^{-(b, c)}=\emptyset$. Actually, if $b=0$ (respectively $c=0$) and $a^{-(b, c)}$ exists, then $a^{-(b, c)}=0$, which is impossible, for $c\neq 0$
(respectively $b\neq 0$).  In the following theorem the set $\A^{-(b, c)}$ will be proved to be open.

\begin{thm}\label{thm4.1}Let $\A$ be a unitary Banach algebra and consider $b$, $c\in\A$. Then $\A^{-(b, c)}$ is an open set.
Furthermore, if $b$, $c\in\hat{\A}\setminus\{0\}$, $a\in\A^{-(b, c)}$ and $e\in\A$ is such that $\parallel e\parallel< \frac{1}{\parallel a^{-(b, c)}\parallel}$,
then $a+e\in \A^{-(b, c)}$ and 
$$
(a+e)^{-(b, c)}= (\uno+ a^{-(b, c)}e)^{-1}a^{-(b,c)}=a^{-(b, c)}(\uno+ea^{-(b, c)})^{-1}.
$$
\end{thm}
\begin{proof}According to what has been said in the first paragraph of this section, it is enough to consider the case $b$, $c\in\hat{\A}\setminus\{0\}$. Recall, in addition, that under this hypothesis,
$a^{-(b, c)}\neq0$. In fact, according to Definition \ref{def1}, $a^{-(b, c)}=0$ implies that $b=0=c$. Moreover, note also that since $\parallel a^{-(b,c)}e\parallel <1$ (respectively  $\parallel e a^{-(b,c)}\parallel <1$),
$\uno+a^{-(b,c)}e\in \A^{-1}$ (respectively $\uno+ea^{-(b,c)}\in \A^{-1}$).

\indent Consider the operators  $L_a$, $L_{a^{-(b,c)}}$, $L_e \in\L(\A)$. According to Proposition \ref{pro3.10} (i),
$L_{a^{-(b,c)}}=(L_a)^{(2)}_{b\A, c^{-1}(0)}$.  Now, since $\A$ is a unitary algebra,
$$
\parallel (L_a)^{(2)}_{b\A, c^{-1}(0)}\parallel \parallel L_e\parallel= \parallel a^{-(b,c)}\parallel \parallel e\parallel <1.
$$
Thus, according to \cite[Lemma 3.4]{DX}, $(L_{a+e})^{(2)}_{b\A, c^{-1}(0)}$ exists and 
$$
(L_{a+e})^{(2)}_{b\A, c^{-1}(0)}=(I+ L_{a^{-(b,c)}}L_e)^{-1}L_{a^{-(b,c)}}=L_{a^{-(b,c)}}(I+L_eL_{a^{-(b,c)}})^{-1}.
$$
However,  
\begin{align*}
&(I+ L_{a^{-(b,c)}}L_e)^{-1}=L_{(\uno+a^{-(b,c)}e)^{-1}},& &(I+ L_eL_{a^{-(b,c)}})^{-1}=L_{(\uno+ea^{-(b,c)})^{-1}}.&
\end{align*}
Consequently, 
$$
(\uno+a^{-(b,c)}e)^{-1}a^{-(b,c)}=a^{-(b,c)}(\uno+ea^{-(b,c)})^{-1}.
$$
\indent Set $f=(\uno+a^{-(b,c)}e)^{-1}a^{-(b,c)}=a^{-(b,c)}(\uno+ea^{-(b,c)})^{-1}$. Since $(L_{a+e})^{(2)}_{b\A, c^{-1}(0)}=L_f$,
according to Proposition \ref{pro3.10} (ii), $(a+e)^{-(b, c)}$ exists and $(a+e)^{-(b, c)}=f$.
\end{proof}

In the following theorem   the case of the inverse along an element will be presented. To this end, given a unitary Banach algebra $\A$, let $\A^{-d}$ be the set of all elements of
$\A$ invertible along $d\in\A$, i.e., 
$$
\A^{-d}=\{a\in\A\colon a^{-d}\hbox{ exists}\}.
$$

\begin{thm}\label{thm4.2}Let $\A$ be a unitary Banach algebra and consider $d\in\A$. Then $\A^{-d}$ is an open set.
Furthermore, if $d\in\hat{\A}\setminus\{0\}$, $a\in\A^{-d}$ and $e\in\A$ is such that $\parallel e\parallel< \frac{1}{\parallel a^{-d}\parallel}$,
then $a+e\in \A^{-d}$ and 
$$
(a+e)^{-d}= (\uno+ a^{-d}e)^{-1}a^{-d}=a^{-d}(\uno+ea^{-d})^{-1}.
$$
\end{thm}
\begin{proof} Note that according to \cite[Proposition 6.1]{D}, $\A^{-d}=\A^{-(d, d)}$. Now apply Theorem \ref{thm4.1}.
\end{proof}

\indent For sake of completeness, the case of the outer inverse with prescribed range and null space will be considered. Let $\T$ and $\S$ be two closed
subspace of the Banach spaces $\X$ and $\Y$, respectivley.  $\L(\X, \Y)^{(2)}_{\T, \S}$ will stand for  the set of all operators defined on $\X$ with values in $\Y$ whose outer inverse with range $\T$ and null space $\S$ exists, i.e.,
$$
\L(\X, \Y)^{(2)}_{\T, \S}=\{A\in\L(\X, \Y)\colon A^{(2)}_{\T, \S} \hbox{ exists}\}.
$$

\begin{thm}\label{thm4.3} Let $\X$ and $\Y$ be two Banach spaces and consider $\T$ and $\S$ two closed subspaces of $\X$ and $\Y$, respectively. Then, 
$\L(\X, \Y)^{(2)}_{\T, \S}$ is an open set.
\end{thm}
\begin{proof} According to \cite[Lemma 1]{LYZW}, if $\T$ or $\S$ is not a complemented subspaces of $\X$ and $\Y$ respectively, then $\L(\X, \Y)^{(2)}_{\T, \S}=\emptyset$. On the other hand, note that
if $A\in\L(\X, \Y)$  is such that $A^{(2)}_{\T, \S}$ exists and $A^{(2)}_{\T, \S}=0$, then $\T=0$ and $\S=\Y$. In addition, in this case,  $\L(\X, \Y)^{(2)}_{0, \Y}=\L(\X, \Y)$. In fact, given $A\in\L(\X, \Y)$,
$A^{(2)}_{0, \Y}=0$. To end the proof, apply \cite[Lemma 3.4]{DX}.
\end{proof} 

\section{Continuity of the $(b, c)$-inverse}

Recall that the notion of the gap between subspaces was used to study the continuity of the Moore-Penrose inverse ({\hspace{-1pt}\cite{V2}) and the Drazin inverse ({\hspace{-1pt}\cite{kr1, V}) in the Banach context.
In this section the aforementioned notion will be used to characterize the continuity of the $(b, c)$-inverse for Banach space operators and Banach algebra elements.
Another notion that will be central to present more characterizations will be the Moore-Penrose inverse. 

First  a particular case will be presented.

\begin{rema}\label{rema500}\rm \noindent (i) Let $\X$ be a Banach space and consider $A$, $B$, $C\in \L (\X)$ such that
$B$ and $C$ are regular, $A$ is $(B, C)$-invertible and $A^{-(B,C)}=0$. Let $(A_n)_{n\in\mathbb{N}}\subset \L(\X)$ be such that 
$(A_n)_{n\in\mathbb{N}}$ converges to $A\in\L(\X)$, and suppose that there exist two sequence of operators
$(B_n)_{n\in\mathbb{N}}$, $(C_n)_{n\in\mathbb{N}}\subset \L(\X)$ such that
for each $n\in\mathbb{N}$, $B_n$  and $C_n$ are regular and $A_n$ is $(B_n, C_n)$-invertible.
Then, $(A_n^{-(B_n, C_n)})_{n\in\mathbb{N}}$ converges to $A^{-(B,C)} (=0)$ if and only if there exists $n_0\in\mathbb{N}$
such that for each $n\ge n_0$,  $A_n^{-(B_n, C_n)}=0$. 
In fact, if  $(A_n^{-(B_n, C_n)})_{n\in\mathbb{N}}$ converges to 0, then $(A_n^{-(B_n, C_n)}A_n)_{n\in\mathbb{N}}$ converges to 0,
and according to \cite[Lemma 3.3]{kr1},  $(\hat{\delta}(\R(A_n^{-(B_n, C_n)}A_n), 0))_{n\in\mathbb{N}}$ converges to 0.
However, according to the definition of the gap between two subspaces (\hspace{-1pt}see \cite[Chapter 2, Section 2, Subsection 1]{K}),
if $\R(A_n^{-(B_n, C_n)}A_n)\neq 0$, then $(\hat{\delta}(\R(A_n^{-(B_n, C_n)}A_n), 0))=1$. Therefore, there exists $n_0 \in\mathbb{N}$
such that for each $n\ge n_0$, $\R(A_n^{-(B_n, C_n)})=\R(A_n^{-(B_n, C_n)}A_n)=0$. As a result, $A_n^{-(B_n, C_n)}=0$ for each $n\ge n_0$. The converse implication is evident.\par
\noindent (ii) Let $\A$ be a Banach algebra and consider $a\in\A$ and $b$, $c\in\hat{\A}$ such that $a$ is $(b, c)$-invertible and $a^{-(b, c)}=0$.
Suppose that there exist three sequence $(a_n)_{n\in\mathbb{N}}\subset \A$ and  $(b_n)_{n\in\mathbb{N}}$, $(c_n)_{n\in\mathbb{N}}\subset \hat{\A}$ such that
for each $n\in\mathbb{N}$, $a_n$ is $(b_n, c_n)$-invertible and $(a_n)_{n\in\mathbb{N}}$ converges to $a$.  Then, statement (i) can be extended to this case, i.e., necessary and sufficient for 
$(a_n^{-(b_n, c_n)})_{n\in\mathbb{N}}$ to converge to $a^{-(b, c)} (=0)$ is that there exists  
$n_0\in\mathbb{N}$ such that for all $n\ge n_0$, $a_n^{-(b_n, c_n)}=0$. Actually, to prove this equivalent condition, it is enough to apply statement (i) to $L_a$, $L_a^{-(L_b, L_c)}\in\L(\A)$ and
$(L_{a_n})_{n\in\mathbb{N}}$, $(L_{a_n}^{-(L_{b_n}, L_{c_n})})_{n\in\mathbb{N}}\subset\L(\A)$, ($L_a^{-(L_b, L_c)}=L_{a^{-(b, c)}}$,
$L_{a_n}^{-(L_{b_n}, L_{c_n})}=L_{a_n^{-(b_n, c_n)}}$, see Proposition \ref{pro3.10} (i)). Note that since $\A$ is a unitary Banach algebra, 
$(a_n)_{n\in\mathbb{N}}$ (respectively $(a_n^{-(b_n, c_n)})_{n\in\mathbb{N}}$) converges to $a$ (respectively to $a^{-(b, c)}$)
if and only if $(L_{a_n})_{n\in\mathbb{N}}$ (respectively $(L_{a_n^{-(b_n, c_n)}})_{n\in\mathbb{N}}$ converges to $L_a$ (respectively to $L_{a^{-(b, c)}}$).\par
\noindent (iii) In the same conditions of statement (ii), note that according to Definition \ref{def1}, $a^{-(b,c)}=0$ implies that $b=0=c$ . Similarly, for each $n\ge n_0$, $b_n=0=c_n$, i.e., 
the fact that $(a_n^{-(b_n, c_n)})_{n\in\mathbb{N}}$ converges to 0 determines the elements $b_n$ and $c_n$ for which $a_n$ is $(b_n, c_n)$-invertible ($n\ge n_0$). 
Naturally, given $a\in\A^{-(0, 0)}$, $a^{-(0, 0)}=0$. \par
\noindent (iv) It is worth noticing that if $\A$ is a unitary Banach algebra, $a\in\A$ and $b$, $c\in\hat{\A}$ are such that $a$ is $(b, c)$-invertible with $a^{-(b, c)}\neq 0$, then $b\neq 0$ and $c\neq 0$
(see Definition \ref{def1}).  
\end{rema}

Now  the continuity of the $(b, c)$-inverse will be studied using the gap between subspaces. Next follows the characterization of the continuity of $(B, C)$-invertible Banach space operators. 

\begin{thm}\label{thm5.2}Let $\X$ be a Banach space and consider $A$, $B$, $C\in \L (\X)$ such that
$B$ and $C$ are regular, $A$ is $(B, C)$-invertible and $A^{-(B,C)}\neq 0$. Suppose that there exist three sequences of 
operators $(A_n)_{n\in\mathbb{N}}$, $(B_n)_{n\in\mathbb{N}}$, $(C_n)_{n\in\mathbb{N}}\subset \L(\X)$ such that
for each $n\in\mathbb{N}$, $B_n$  and $C_n$ are regular and $A_n$ is $(B_n, C_n)$-invertible.
If  $(A_n)_{n\in\mathbb{N}}$ converges to $A$, then the following statements are equivalent.
\begin{enumerate}[{\rm (i)}] 
\item The sequence $(A_n^{-(B_n, C_n)})_{n\in\mathbb{N}}$ converges to $A^{-(B,C)}$.

\item The sequence $(A_n^{-(B_n, C_n)}A_n)_{n\in\mathbb{N}}$  \rm(\it respectively $(A_nA_n^{-(B_n, C_n)})_{n\in\mathbb{N}}$\rm) \it
converges to $A^{-(B, C)}A$ \rm(\it respectively  to $AA^{-(B, C)}$\rm).\it 

\item The sequence $(A_n^{-(B_n, C_n)}A_n)_{n\in\mathbb{N}}$ \rm(\it respectively $(\hat{\delta}(\N(C_n), \N(C)))_{n\in\mathbb{N}}$\rm) \it converges to $A^{-(B, C)}A$ 
 \rm(\it respectively  to 0\rm).\it 
\item The sequence $(A_nA_n^{-(B_n, C_n)})_{n\in\mathbb{N}}$  \rm(\it respectively $(\hat{\delta}(\R(B_n), \R(B)))_{n\in\mathbb{N}}$\rm) \it
converges to $AA^{-(B, C)}$  \rm(\it respectively to 0\rm).\it

\item The sequences $(\hat{\delta}(\R(B_n), \R(B)))_{n\in\mathbb{N}}$ and $(\hat{\delta}(\N(C_n), \N(C)))_{n\in\mathbb{N}}$
converge to $0$.
\item The sequences $(\hat{\delta}(\R(A_n^{-(B_n, C_n)}), \R(A^{-(B, C)})))_{n\in\mathbb{N}}$ and $(\hat{\delta}(\N(A_n^{-(B_n, C_n)}), \N(A^{-(B, C)})))_{n\in\mathbb{N}}$
converge to $0$.

\end{enumerate}
\end{thm}
\begin{proof}It is evident that statement (i) implies statement (ii).\par

\indent Observe that, since $A^{-(B, C)}$ is an outer inverse of $A$ and $A_n^{-(B_n, C_n)}$ is an outer inverse of $A_n$ ($n\in\mathbb{N}$),
according to Theorem \ref{thm3.1},
\begin{align*}
&\R(A^{-(B, C)}A)=\R(A^{-(B, C)})=\R(B),& &\N(AA^{-(B, C)})=\N(A^{-(B, C)})=\N(C),&\\
&\R(A_n^{-(B_n, C_n)}A_n)=\R(A_n^{-(B_n, C_n)})=\R(B_n),& &\N(A_nA_n^{-(B_n, C_n)})=\N(A_n^{-(B_n, C_n)})=\N(C_n).&
\end{align*}
Consequently, according to \cite[Lemma 3.3]{kr1}, statment (ii) implies statement (iii), which in turn implies statement (v).
In addition, applying again \cite[Lemma 3.3]{kr1}, statement (ii) also implies statement (iv), which in turn implies statement (v). Note also that statement (vi) 
is an equivalent formulation of statement (v) (Theorem \ref{thm3.2}).

\indent Suppose that statement (vi) holds. According to Theorem \ref{thm3.2}, 
\begin{align*}
&A_n^{-(B_n, C_n)}={A_n}^{(2)}_{\R(B_n), N(C_n)},& &A^{-(B, C)}=A^{(2)}_{\R(B), N(C)}.&
\end{align*} 
\noindent Let $\kappa=\parallel A\parallel\parallel A^{-(B, C)}\parallel$ and consider $n_0\in\mathbb{N}$ such that
for all $n\ge n_0$, 
\begin{align*}
&u_n=\hat{\delta}(\N(C_n), \N(C))<\frac{1}{3+\kappa},&
&v_n=\hat{\delta}(\R(B_n), \R(B))<\frac{1}{(1+\kappa)^2},&\\
&z_n=\parallel A^{-(B, C)}\parallel\parallel A-A_n\parallel<\frac{2\kappa}{(1+\kappa)(4+\kappa)}.& & &\\
\end{align*}

\noindent Thus, according to \cite[Theorem 3.5]{DX},
$$
\parallel A_n^{-(B_n, C_n)}-A^{-(B, C)}\parallel \le\frac{(1+\kappa)(v_n+u_n)+(1+u_n)z_n}{1-(1+\kappa) v_n-\kappa u_n-(1+u_n)z_n}\parallel A^{-(B, C)}\parallel,
$$
which implies statement (i).
\end{proof}

Next the Banach algebra case will be considered.

\begin{thm}\label{thm5.3}Let $\A$ be a unitary Banach algebra and consider $a\in\A$ and $b$, $c\in\hat{\A}$ such that $a$ is $(b, c)$-invertible and $a^{-(b, c)}\neq 0$.
Suppose that there exist three sequences $(a_n)_{n\in\mathbb{N}}\subset\A$ and  $(b_n)_{n\in\mathbb{N}}$, $(c_n)_{n\in\mathbb{N}}\subset \hat{\A}$ such that 
$a_n$ is $(b_n, c_n)$-invertible, for each $n\in\mathbb{N}$.
If  $(a_n)_{n\in\mathbb{N}}$ converges to $a$, then the following statements are equivalent.
\begin{enumerate}[{\rm (i)}] 
\item The sequence $(a_n^{-(b_n, c_n)})_{n\in\mathbb{N}}$ converges to $a^{-(b, c)}$.

\item  The sequences $(a_n^{-(b_n, c_n)}a_n)_{n\in\mathbb{N}}$  and $(a_na_n^{-(b_n, c_n)})_{n\in\mathbb{N}}$
converge to $a^{-(b, c)}a$ and $aa^{-(b, c)}$, respectively. 

\item  The sequence $(a_n^{-(b_n, c_n)}a_n)_{n\in\mathbb{N}}$ \rm(\it respectively $(\hat{\delta}((c_n)^{-1}(0), c^{-1}(0)))_{n\in\mathbb{N}}$\rm) \it
converges to $a^{-(b, c)}a$ \rm(\it respectively to 0\rm)\it.

\item  The sequence $(a_na_n^{-(b_n, c_n)})_{n\in\mathbb{N}}$ \rm(\it respectively  $(\hat{\delta}(b_n\A, b\A))_{n\in\mathbb{N}}$\rm) \it
converges to $aa^{-(b, c)}$ \rm(\it respectively to 0\rm)\it.

\item The sequences $(\hat{\delta}(b_n\A, b\A))_{n\in\mathbb{N}}$ and $(\hat{\delta}((c_n)^{-1}(0), c^{-1}(0)))_{n\in\mathbb{N}}$
converge to $0$.

\item  The sequences $(\hat{\delta}(a_n^{-(b_n, c_n)}\A, a^{-(b, c)}\A))_{n\in\mathbb{N}}$ and $(\hat{\delta}((a_n^{-(b_n, c_n)})^{-1}(0), (a^{-(b, c)})^{-1}(0)))_{n\in\mathbb{N}}$
converge to $0$.\par

\item The sequence $(a_na_n^{-(b_n, c_n)})_{n\in\mathbb{N}}$ \rm(\it respectively $(\hat{\delta}((b_n)_{-1}(0), b_{-1}(0)))_{n\in\mathbb{N}}$\rm) \it
converges to $aa^{-(b, c)}$ \rm(\it respectively to 0\rm)\it.

\item The sequence $(a_n^{-(b_n, c_n)}a_n)_{n\in\mathbb{N}}$ \rm(\it respectively $(\hat{\delta}(\A c_n, \A c))_{n\in\mathbb{N}}$\rm) \it
converges to $a^{-(b, c)}a$ \rm(\it respectively to 0\rm)\it.

\item The sequences $(\hat{\delta}(\A c_n, \A c))_{n\in\mathbb{N}}$ and $(\hat{\delta}((b_n)_{-1}(0), b_{-1}(0)))_{n\in\mathbb{N}}$
converge to $0$.

\item  The sequences $(\hat{\delta}(\A a_n^{-(b_n, c_n)}, \A a^{-(b, c)}))_{n\in\mathbb{N}}$ and $(\hat{\delta}((a_n^{-(b_n, c_n)})_{-1}(0), (a^{-(b, c)})_{-1}(0)))_{n\in\mathbb{N}}$
converge to $0$.\par
\end{enumerate}
\end{thm}

\begin{proof} Recall that, according to Proposition \ref{pro3.10} (i), 
\begin{align*}
&L_{a^{-(b, c)}}=L_a^{-(L_b, L_c)}=(L_a)^{(2)}_{b\A, c^{-1}(0)},& &L_{a_n^{-(b_n, c_n)}}=L_{a_n}^{-(L_{b_n}, L_{c_n})}=(L_{a_n})^{(2)}_{b_n\A, {c_n}^{-1}(0)},&\\
\end{align*}
for each $n\in\mathbb{N}$. In addition, note that since $\A$ is a unitary Banach algebra, $(a_n)_{n\in\mathbb{N}}$ converges to $a$
if and only if $(L_{a_n})_{n\in\mathbb{N}}$ converges to $L_a$. Similarly,  a necessary and sufficient condition for $(a_n^{-(b_n, c_n)})_{n\in\mathbb{N}}$ (respectively for 
$(a_n^{-(b_n, c_n)}a_n)_{n\in\mathbb{N}}$ and  $(a_na_n^{-(b_n, c_n)})_{n\in\mathbb{N}}$)
to converge to $a^{-(b, c)}$ (respectively to $a^{-(b, c)}a$ and  $aa^{-(b, c)}$) is that $(L_{a_n^{-(b_n, c_n)}})_{n\in\mathbb{N}}$ (respectively $(L_{a_n^{-(b_n, c_n)}}L_{a_n})_{n\in\mathbb{N}}$ and
 $(L_{a_n}L_{a_n^{-(b_n, c_n)}})_{n\in\mathbb{N}}$)
 converges to $L_{a^{-(b, c)}}$ (respectively to $L_{a^{-(b, c)}}L_a$ and $L_aL_{a^{-(b, c)}}$).

Now, to prove the equivalence between statements (i)-(vi), apply  Theorem \ref{thm5.2} to $L_a$ and $(L_{a_n})_{n\in\mathbb{N}}$.
Recall that according to the proof of Proposition \ref{pro3.10}, $c^{-1}(0)=(a^{-(b, c)})^{-1}(0)$ and $c_n^{-1}(0)=(a_n^{-(b_n, c_n)})^{-1}(0)$.

A similar argument, using in particular Proposition \ref{pro3.10} (iii) and Theorem \ref{thm5.2}, proves the equivalence between statements (i), (ii) and  (vii)-(x).
\end{proof}
Next the continuity of the $(b, c)$-inverse will be characterized using the Moore-Penrose inverse.

\begin{thm}\label{thm5.4}Let $\A$ be a unitary Banach algebra and consider $a\in\A$ and $b$, $c\in\A^\dag$ such that $a$ is $(b, c)$-invertible and $a^{-(b, c)}\neq 0$.
Suppose that there exist three sequences $(a_n)_{n\in\mathbb{N}}\subset\A$ and  $(b_n)_{n\in\mathbb{N}}$, $(c_n)_{n\in\mathbb{N}}\subset \A^\dag$ such that 
$a_n$ is $(b_n, c_n)$-invertible, for each $n\in\mathbb{N}$.
If  $(a_n)_{n\in\mathbb{N}}$ converges to $a$, then the following statements are equivalent.
\begin{enumerate}[{\rm (i)}] 
\item The sequence $(a_n^{-(b_n, c_n)})_{n\in\mathbb{N}}$ converges to $a^{-(b, c)}$.

\item  The sequence $(a_n^{-(b_n, c_n)}a_n)_{n\in\mathbb{N}}$ converges to $a^{-(b, c)}a$ and the sequences 
\begin{align*}
&(c^\dag c(\uno-c_n^\dag c_n))_{n\in\mathbb{N}},&
&(c_n^\dag c_n (\uno -c^\dag c))_{n\in\mathbb{N}}\\
\end{align*}
converge to 0.

\item  The sequence $(a_na_n^{-(b_n, c_n)})_{n\in\mathbb{N}}$ converges to $aa^{-(b, c)}$ and the sequences 
\begin{align*}
&((\uno-bb^\dag)b_nb_n^\dag)_{n\in\mathbb{N}},& &( (\uno-b_nb_n^\dag)bb^\dag)_{n\in\mathbb{N}}&\\
\end{align*}
converge to 0.

\item The sequences 
\begin{align*}
&( (\uno-bb^\dag)b_nb_n^\dag)_{n\in\mathbb{N}},& &( (\uno-b_nb_n^\dag)bb^\dag)_{n\in\mathbb{N}},&\\ 
&( c^\dag c(\uno-c_n^\dag c_n))_{n\in\mathbb{N}},& 
&( c_n^\dag c_n (\uno -c^\dag c))_{n\in\mathbb{N}}&\\  
\end{align*}
converge to 0.

\item The sequence $(a_na_n^{-(b_n, c_n)})_{n\in\mathbb{N}}$ converges to $aa^{-(b, c)}$  and the sequences
\begin{align*}
&( bb^\dag (\uno-b_nb_n^\dag))_{n\in\mathbb{N}},& &( b_nb_n^\dag (\uno-bb^\dag))_{n\in\mathbb{N}}&\\
\end{align*}
converge to 0.

\item The sequence $(a_n^{-(b_n, c_n)}a_n)_{n\in\mathbb{N}}$ converge to $a^{-(b, c)}a$ and the sequences 
\begin{align*}
&(  (\uno -c^\dag c)c_n^\dag c_n)_{n\in\mathbb{N}},& &(  (\uno -c_n^\dag c_n)c^\dag c)_{n\in\mathbb{N}}&\\
\end{align*}
converge to 0.

\item The sequences 
\begin{align*}
&( bb^\dag (\uno-b_nb_n^\dag))_{n\in\mathbb{N}},&  &( b_nb_n^\dag (\uno-bb^\dag))_{n\in\mathbb{N}},&\\
&(  (\uno -c^\dag c)c_n^\dag c_n)_{n\in\mathbb{N}},& &(  (\uno -c_n^\dag c_n)c^\dag c)_{n\in\mathbb{N}}&\\
\end{align*}
converge to $0$.

\item The sequences $(b_nb_n^\dag)_{n\in\mathbb{N}}$ and $(c_n^\dag c_n)_{n\in\mathbb{N}}$ converge to $bb^\dag$ and $c^\dag c$, respectively. 
\end{enumerate}
\end{thm}
\begin{proof}Note that 
\begin{align*}
&b\A= bb^\dag A,& &b_n\A=b_nb_n^\dag\A,& &c^{-1}(0)=(\uno-c^\dag c)\A,& &c_n^{-1}(0)=(\uno-c_n^\dag c_n)\A,&\\
&\A c=\A c^\dag c,& &\A c_n=\A c_n^\dag c,& &b_{-1}(0)=\A (\uno-bb^\dag),& &(b_n)_{-1}(0)=\A (\uno-b_nb_n^\dag).& \\
\end{align*}
Since $\A$ is a unitary Banach algebra, according to \cite[Lemma 2.2]{V2},
\begin{align*}
&\hat{\delta}(b_n\A, b\A)=\hbox{\rm max}\{ \parallel (\uno-bb^\dag)b_nb_n^\dag\parallel,  \parallel (\uno-b_nb_n^\dag)bb^\dag\parallel\},&\\
&\hat{\delta}((c_n)^{-1}(0), c^{-1}(0))=\hbox{\rm max}\{ \parallel c^\dag c(\uno-c_n^\dag c_n)\parallel,  \parallel c_n^\dag c_n (\uno -c^\dag c\parallel\},&\\
&\hat{\delta}((b_n)_{-1}(0), b_{-1}(0))=\hbox{\rm max}\{ \parallel bb^\dag (\uno-b_nb_n^\dag)\parallel,  \parallel b_nb_n^\dag (\uno-bb^\dag)\}\parallel\},&\\
&\hat{\delta}(\A c_n, \A c)=\hbox{\rm max}\{\parallel  (\uno -c^\dag c)c_n^\dag c_n\parallel, \parallel  (\uno -c_n^\dag c_n)c^\dag c\parallel\}.&\\
\end{align*}

 \indent Now, to prove the equivalence among statements (i)-(vii), apply Theorem \ref{thm5.3} and use the identities that have been proved.

Suppose that statement (i) holds. Observe that
\begin{align*}
b_nb_n^\dag&=bb^\dag b_nb_n^\dag bb^\dag + bb^\dag b_nb_n^\dag (\uno-bb^\dag) + (\uno-bb^\dag) b_nb_n^\dag bb^\dag+ (\uno-bb^\dag)b_nb_n^\dag(\uno-bb^\dag).\\
\end{align*}
Thus, according to statements (iv) and (vii), if 
$$
f_n=bb^\dag b_nb_n^\dag (\uno-bb^\dag) + (\uno-bb^\dag) b_nb_n^\dag bb^\dag+ (\uno-bb^\dag)b_nb_n^\dag(\uno-bb^\dag),
$$
then the sequence $(f_n)_{n\in\mathbb{N}}$ converges to 0. In addition, according to statement (vii), the sequence $(bb^\dag (\uno-b_nb_n^\dag)bb^\dag)_{n\in\mathbb{N}}$ converges to 0, which implies that $(bb^\dag b_nb_n^\dag bb^\dag)_{n\in\mathbb{N}}$
convergs to $bb^\dag$. Therefore, $(b_nb_n^\dag)_{n\in\mathbb{N}}$ converges to $bb^\dag$. A similar argument proves that  $(c_n^\dag c_n)_{n\in\mathbb{N}}$ converges to $c^\dag c$.

\indent It is evident that statement (viii) implies statement (vii).
\end{proof}

In the following corollary, a particular case will be presented.

\begin{cor}\label{cor5.5}
Let $\A$ be a unitary Banach algebra and consider $a\in\A$ and $b$, $c\in\A^\dag$ such that $a$ is $(b, c)$-invertible and $a^{-(b, c)}\neq 0$.
Suppose that there exist three sequences $(a_n)_{n\in\mathbb{N}}\subset\A$ and  $(b_n)_{n\in\mathbb{N}}$, $(c_n)_{n\in\mathbb{N}}\subset \A^\dag$ such that 
$a_n$ is $(b_n, c_n)$-invertible, for each $n\in\mathbb{N}$. Suppose, in addition, that the sequences $(b_n)_{n\in\mathbb{N}}$, $(c_n)_{n\in\mathbb{N}}$,
$(b_n^\dag)_{n\in\mathbb{N}}$ and $(c_n^\dag)_{n\in\mathbb{N}}$ converge to $b$, $c$, $b^\dag$ and $c^\dag$, respectively.
Then, if  $(a_n)_{n\in\mathbb{N}}$ converges to $a$, the sequence $(a_n^{-(b_n, c_n)})_{n\in\mathbb{N}}$ converges to $a^{-(b, c)}$.
\end{cor}
\begin{proof} Apply Theorem \ref{thm5.4} (viii).
\end{proof}

In the context of $C^*$-algebras, Theorem \ref{thm5.4} and Corollary \ref{cor5.5} can be reformulated as follows.

\begin{thm}\label{thm5.6} Let $\A$ be a $C^*$-algebra and consider $a\in\A$ and $b$, $c\in\hat{\A}$ such that $a$ is $(b, c)$-invertible and $a^{-(b, c)}\neq 0$.
Suppose that there exist three sequences $(a_n)_{n\in\mathbb{N}}\subset\A$ and  $(b_n)_{n\in\mathbb{N}}$, $(c_n)_{n\in\mathbb{N}}\subset\hat{\A}$ such that 
$a_n$ is $(b_n, c_n)$-invertible, for each $n\in\mathbb{N}$.
If  $(a_n)_{n\in\mathbb{N}}$ converges to $a$, then the following statements are equivalent.
\begin{enumerate}[{\rm (i)}] 
\item The sequence $(a_n^{-(b_n, c_n)})_{n\in\mathbb{N}}$ converges to $a^{-(b, c)}$.

\item  The sequences $(a_n^{-(b_n, c_n)}a_n)_{n\in\mathbb{N}}$  and $(c_n^\dag c_n)_{n\in\mathbb{N}}$ converge to $a^{-(b, c)}a$ and $c^\dag c$, respectively.

\item  The sequences $(a_na_n^{-(b_n, c_n)})_{n\in\mathbb{N}}$ and $(b_nb_n^\dag)_{n\in\mathbb{N}}$ converge to $aa^{-(b, c)}$ and 
$bb^\dag$, respectively.

\item The sequences $(b_nb_n^\dag)_{n\in\mathbb{N}}$ and $(c_n^\dag c_n)_{n\in\mathbb{N}}$ converge to $bb^\dag$ and $c^\dag c$, respectively.

\hspace*{\dimexpr\linewidth-\textwidth\relax}
\begin{minipage}[t]{\textwidth}
\noindent In addition, if the sequence $(b_n)_{n\in\mathbb{N}}$ converges to $b$, then statement {\rm (iii)} is equivalent to:
\end{minipage}
\item the sequences $(a_na_n^{-(b_n, c_n)})_{n\in\mathbb{N}}$ and $(b_n^\dag)_{n\in\mathbb{N}}$ converge to $aa^{-(b, c)}$ and 
$b^\dag$, respectively.

\hspace*{\dimexpr\linewidth-\textwidth\relax}
\begin{minipage}[t]{\textwidth}
\noindent  Moreover, if the sequence $(c_n)_{n\in\mathbb{N}}$ converge to $c$, then statement {\rm (ii)} is equivalent to:
\end{minipage}
\item the sequences $(a_n^{-(b_n, c_n)}a_n)_{n\in\mathbb{N}}$  and $(c_n^\dag )_{n\in\mathbb{N}}$ converge to $a^{-(b, c)}a$ and $c^\dag$, respectively.

\hspace*{\dimexpr\linewidth-\textwidth\relax}
\begin{minipage}[t]{\textwidth}
\noindent Furthermore, if the sequences $(b_n)_{n\in\mathbb{N}}$ and $(c_n)_{n\in\mathbb{N}}$ converge to $b$  and $c$, respectively, then statement {\rm (iv)} is equivalent to: 
\end{minipage}
\item the sequences $(b_n^\dag)_{n\in\mathbb{N}}$ and $(c_n^\dag )_{n\in\mathbb{N}}$ converge to $b^\dag$ and $c^\dag$, respectively.
\end{enumerate}
\end{thm}
\begin{proof}Note that since
\begin{align*}
&((\uno-bb^\dag)b_nb_n^\dag)^*=b_nb_n^\dag(\uno-bb^\dag),& &( (\uno-b_nb_n^\dag)bb^\dag)^*=bb^\dag(\uno-b_nb_n^\dag),&\\
\end{align*}
according to the proof of  statement (vii) implies statement (viii) in Theorem \ref{thm5.4}, statement (iii) is equivalent to Theorem \ref{thm5.4} (iii).

Now observe that, 
\begin{align*}
&(c^\dag c(\uno-c_n^\dag c_n))^*=(\uno-c_n^\dag c_n)c^\dag c,&
&(c_n^\dag c_n (\uno -c^\dag c))^*=(\uno -c^\dag c)c_n^\dag c_n.\\
\end{align*}
Then, a similar argument proves that statement (ii) is equivalent to Theorem \ref{thm5.4} (ii).

To prove statements (v)-(vii), apply \cite[Theorem 1.6]{kol}.
\end{proof}

Since the inverse along an element is a particular case of the $(b, c)$-inverse, it is possible to apply the results of this section 
to obtain characterizations of the continuity of the inverse along an element for Banach space operators and for Banach algebra and $C^*$-algebra elements.
In fact, given a Banach algebra (or a $C^*$-algebra) $\A$, $a\in\A$ and $d\in\hat{\A}$,
to state the aforementioned characterizations, it is enough to apply the corresponding result to the $(d, d)$-inverse (see section 2);
in the Banach operator case it is possible to proceed in a similar way. 
In order not to extend unnecessarily this work, these results will not be stated; the details are left to the interested reader.

To end this section, as an application, characterizations of the continuity of the Moore-Penrose inverse
in the contexts of Banch algebras and Banach space operators  will be given;  in the former frame, compare with \cite[Theorem 2.5]{V2}.

\begin{cor}\label{cor5.7}Let $\A$ be a unitary Banach algebra and consider $a\in\A^\dag\setminus 0$.
Suppose that there exists a sequence $(a_n)_{n\in\mathbb{N}}\subset\A^\dag$ such that 
$(a_n)_{n\in\mathbb{N}}$ converges to $a$. Then the following statements are equivalent.
\begin{enumerate}[{\rm (i)}] 
\item The sequence $(a_n^\dag)_{n\in\mathbb{N}}$ converges to $a^{\dag}$.

\item  The sequence $(a_n^\dag a_n)_{n\in\mathbb{N}}$ converges to $a^\dag a$ and the sequences 
\begin{align*}
&(a a^\dag(\uno-a_n a_n^\dag))_{n\in\mathbb{N}},&
&(a_n a_n^\dag (\uno -a a^\dag))_{n\in\mathbb{N}}\\
\end{align*}
converge to 0.

\item  The sequence $(a_na_n^\dag)_{n\in\mathbb{N}}$ converges to $aa^\dag$ and the sequences 
\begin{align*}
&((\uno-a^\dag a)a^\dag_na_n)_{n\in\mathbb{N}},& &( (\uno-a^\dag_na_n)a^\dag a)_{n\in\mathbb{N}}&\\
\end{align*}
converge to 0.

\item The sequences 
\begin{align*}
&( (\uno-a^\dag a)a^\dag_na_n)_{n\in\mathbb{N}},& &( (\uno-a_n^\dag a_n)a^\dag a)_{n\in\mathbb{N}},&\\ 
&( a a^\dag(\uno-a_n a_n^\dag))_{n\in\mathbb{N}},& 
&( a_n a_n^\dag (\uno -a a^\dag))_{n\in\mathbb{N}}&\\  
\end{align*}
converge to 0.

\item The sequence $(a_na_n^\dag)_{n\in\mathbb{N}}$ converges to $aa^\dag$  and the sequences
\begin{align*}
&( a^\dag a (\uno-a^\dag_na_n))_{n\in\mathbb{N}},& &( a^\dag_n a_n (\uno-a^\dag a))_{n\in\mathbb{N}}&\\
\end{align*}
converge to 0.

\item The sequence $(a_n^\dag a_n)_{n\in\mathbb{N}}$ converge to $a^\dag a$ and the sequences 
\begin{align*}
&(  (\uno -a a^\dag)a_n a_n^\dag)_{n\in\mathbb{N}},& &(  (\uno -a_n a_n^\dag)a a^\dag)_{n\in\mathbb{N}}&\\
\end{align*}
converge to 0.

\item The sequences 
\begin{align*}
&( a^\dag a (\uno-a^\dag_n a_n))_{n\in\mathbb{N}},&  &( a_n^\dag a_n (\uno-a^\dag a))_{n\in\mathbb{N}},&\\
&(  (\uno -a a^\dag)a_n a^\dag_n)_{n\in\mathbb{N}},& &(  (\uno -a_n a_n^\dag)a a^\dag)_{n\in\mathbb{N}}&\\
\end{align*}
converge to $0$.

\item The sequences $(a_n^\dag a_n)_{n\in\mathbb{N}}$ and $(a_n a^\dag_n)_{n\in\mathbb{N}}$ converge to $a^\dag a$ and $a a^\dag$, respectively. 
\end{enumerate}
\end{cor}
\begin{proof} According to \cite[Proposition 6.1]{D}, given $x\in\A^\dag$, $x^\dag=x^{-(x^\dag, x^\dag)}$. To conclude the proof, apply 
Theorem \ref{thm5.4} to $a$, $b=c=a^\dag$, $a_n$ and $b_n=c_n=a_n^\dag$ ($n\in\mathbb{N}$). Note that if $x\in\A^\dag$, then $x^\dag\in A^\dag$ and 
$(x^\dag)^\dag=x$.
\end{proof}

\begin{rema}\label{rema5.8}\rm In the same conditions of Corollary  \ref{cor5.7}, note that according to \cite[Lemma 2.2]{V2},
\begin{align*}
&\hat{\delta}(\R(L_{a_n^\dag}), \R(L_{a^\dag}))=\hat{\delta}(a_n^\dag\A, a^\dag\A)=\hbox{\rm max}\{ \parallel (\uno-a^\dag a)a_n^\dag a_n\parallel,  \parallel (\uno-a_n^\dag a_n)a^\dag a\}\parallel\},&\\
&\hat{\delta}(\N(L_{a_n^\dag}), \N(L_{a^\dag}))=\hat{\delta}((a_n^\dag)^{-1}(0), (a^\dag)^{-1}(0))=\hbox{\rm max}\{ \parallel aa^\dag (\uno-a_na_n^\dag)\parallel,  \parallel a_na_n^\dag  (\uno -a a^\dag) \parallel\},&\\
&\hat{\delta}(\N(R_{a_n^\dag}), \N(R_{a^\dag}))=\hat{\delta}((a_n^\dag)_{-1}(0), (a^\dag)_{-1}(0))=\hbox{\rm max}\{ \parallel a^\dag a(\uno-a_n^\dag a_n)\parallel,  \parallel a_n^\dag a_n (\uno-a^\dag a) \parallel\},&\\
&\hat{\delta}(\R(R_{a_n^\dag}), \R(R_{a^\dag}))=\hat{\delta}(\A a_n^\dag, \A a^\dag)=\hbox{\rm max}\{\parallel  (\uno -a a^\dag )a_n a_n^\dag \parallel, \parallel  (\uno -a_na_n^\dag)a a^\dag \parallel\}.&\\
\end{align*}
Compare Corollary \ref{cor5.7} with \cite[Theorem 2.5]{V2}.
\end{rema}

\begin{cor}\label{cor5.9}Let $\X$ be a Banach space and consider a Moore-Penrose invertible operator $A\in \L (\X)$, $A\neq 0$. Suppose that there exists a sequence
of operators $(A_n)_{n\in\mathbb{N}}$ such that
for each $n\in\mathbb{N}$, $A_n$ is Moore-Penrose invertible.
If  $(A_n)_{n\in\mathbb{N}}$ converges to $A$, then the following statements are equivalent.

\begin{enumerate}[{\rm (i)}] 
\item The sequence $(A_n^\dag)_{n\in\mathbb{N}}$ converges to $A^\dag$.

\item The sequence $(A_n^\dag A_n)_{n\in\mathbb{N}}$  \rm(\it respectively $(A_nA_n^\dag)_{n\in\mathbb{N}}$\rm) \it
converges to $A^\dag A$ \rm(\it respectively  to $AA^\dag $\rm).\it 

\item The sequence $(A_n^\dag A_n)_{n\in\mathbb{N}}$ \rm(\it respectively $(\hat{\delta}(\N(A_n^\dag), \N(A^\dag)))_{n\in\mathbb{N}}$\rm) \it converges to $A^\dag A$ 
 \rm(\it respectively  to 0\rm).\it 
\item The sequence $(A_nA_n^\dag)_{n\in\mathbb{N}}$  \rm(\it respectively $(\hat{\delta}(\R(A^\dag_n), \R(A^\dag)))_{n\in\mathbb{N}}$\rm) \it
converges to $AA^\dag$  \rm(\it respectively to 0\rm).\it

\item The sequences $(\hat{\delta}(\R(A_n^\dag), \R(A^\dag)))_{n\in\mathbb{N}}$ and $(\hat{\delta}(\N(A^\dag_n), \N(A^\dag)))_{n\in\mathbb{N}}$
converge to $0$.
\end{enumerate}
\end{cor}
\begin{proof}  Proceed as in the proof of Corollary \ref{cor5.7} and apply Theorem \ref{thm5.2}.
\end{proof}


\section{Differentiation of the $(b, c)$-inverse}

Next follows the first characterization of this section. Observe that if $\A$ is a Banach algebra, $J\subseteq \mathbb{R}$ and there exist  functions ${\mathbf a}\colon J\to \A$ and  ${\mathbf b}$, 
${\mathbf c}\colon J\to \hat{\A}$ such that for each $t\in J$, ${\mathbf a}(t)^{-({\mathbf b(t)}, {\mathbf c(t)})}$ exists, then  ${\mathbf a}^{-({\mathbf b }, {\mathbf c})}\colon J\to \A$
is the following funciton:  ${\mathbf a}^{-({\mathbf b }, {\mathbf c})}(t)= {\mathbf a(t)}^{-({\mathbf b(t) }, {\mathbf c(t)})}$.

\begin{thm}\label{thm6.1} Let $\A$ be a Banach algebra and consider a function ${\mathbf a}\colon J\to \A$, where $J\subseteq\mathbb{R}$ is an open set.
Let ${\mathbf b}$, ${\mathbf c}\colon J\to \hat{\A}$ be two functions such that for each $t\in J$, ${\mathbf a}(t)^{-({\mathbf b(t)}, {\mathbf c(t)})}$ exists. Suppose that 
there exist functions ${\mathbf g}$, ${\mathbf h}\colon J\to \A$ such that for each $t\in J$, ${\mathbf g}(t)\in {\mathbf b}(t)\{1\}$ and  ${\mathbf h}(t)\in {\mathbf c}(t)\{1\}$.
Then, given $t_0\in J$, if the functions ${\mathbf a}$, ${\mathbf b}{\mathbf g}$, ${\mathbf h}{\mathbf c}\colon J\to \A$ are differentiable at $t_0$, the following statements are equivalent.
\begin{enumerate}[{\rm (i)}] 
\item The funciton ${\mathbf a}^{-({\mathbf b }, {\mathbf c})}\colon J\to \A$ is continuous at $t_0$.
\item  The funciton ${\mathbf a}^{-({\mathbf b}, {\mathbf c})}\colon J\to \A$ is differentiable at $t_0$.
\end{enumerate}
\noindent Moreover, in this case,
\begin{align*}
({\mathbf a}^{-({\mathbf b }, {\mathbf c})})'(t_0)=&{\mathbf a}(t_0)^{-({\mathbf b(t_0)}, {\mathbf c(t_0)})}({\mathbf h}{\mathbf c})'(t_0){\mathbf a}(t_0)^{-({\mathbf b(t_0)}, {\mathbf c(t_0)})}\\
&-(\uno-{\mathbf a}(t_0)^{-({\mathbf b(t_0)}, {\mathbf c(t_0)})}{\mathbf a}(t_0))({\mathbf b}{\mathbf g})'(t_0){\mathbf a}(t_0)^{-({\mathbf b(t_0)}, {\mathbf c(t_0)})}\\
& -{\mathbf a}(t_0)^{-({\mathbf b(t_0)}, {\mathbf c(t_0)})}{\mathbf a}'(t_0){\mathbf a}(t_0)^{-({\mathbf b(t_0)}, {\mathbf c(t_0)})}.\\
\end{align*}
\end{thm}
\begin{proof} It is enough to prove that statement (i) implies statement (ii).  This proof can be deduced from \cite[Corollary 7.12]{b2}. In fact, according to this result,
\begin{align*}
{\mathbf a}(t)^{-({\mathbf b(t)}, {\mathbf c(t)})}-{\mathbf a}(t_0)^{-({\mathbf b(t_0)}, {\mathbf c(t_0)})}&={\mathbf a}(t)^{-({\mathbf b(t)}, {\mathbf c(t)})}({\mathbf h}{\mathbf c}(t)-{\mathbf h}{\mathbf c}(t_0))
{\mathbf a}(t_0)^{-({\mathbf b(t_0)}, {\mathbf c(t_0)})}\\
&-(\uno-{\mathbf a}(t)^{-({\mathbf b(t)}, {\mathbf c(t)})}{\mathbf a}(t))({\mathbf b}{\mathbf g}(t)-{\mathbf b}{\mathbf g}(t_0)){\mathbf a}(t_0)^{-({\mathbf b(t_0)}, {\mathbf c(t_0)})}\\
&+ {\mathbf a}(t)^{-({\mathbf b(t)}, {\mathbf c(t)})}({\mathbf a}(t_0)-{\mathbf a(t)}){\mathbf a}(t_0)^{-({\mathbf b(t_0)}, {\mathbf c(t_0)})}.\\
\end{align*}
\noindent Now divide each term by $t-t_0$ and note that the limt leads the formula of $({\mathbf a}^{-({\mathbf b }, {\mathbf c})})'(t_0)$.
\end{proof}

When the function ${\mathbf b}$ and ${\mathbf c}$ in Theorem \ref{thm6.1} are such that ${\mathbf b}(J)\subseteq \A^\dag$ and ${\mathbf c}(J)\subseteq \A^\dag$,
the aforementioned result can be reformulated as follows. Note first that if $\A$ is a Banach algebra and ${\mathbf x}\colon J\to \A^\dag$ is a function ($J\subseteq\mathbb{R}$), then ${\mathbf x}^\dag\colon J\to \A$
denotes the following function:  ${\mathbf x}^\dag(t)=  ({\mathbf x}(t))^\dag$.

\begin{cor}\label{cor6.3} Let $\A$ be a Banach algebra and consider a function ${\mathbf a}\colon J\to \A$, where $J\subseteq\mathbb{R}$ is an open set.
Let ${\mathbf b}$, ${\mathbf c}\colon J\to \A^\dag$ be two functions such that for each $t\in J$, ${\mathbf a}(t)^{-({\mathbf b(t)}, {\mathbf c(t)})}$ exists. 
Then, given $t_0\in J$, if the functions ${\mathbf a}$, ${\mathbf b}{\mathbf b}^\dag$, ${\mathbf c}^\dag{\mathbf c}\colon J\to \A$ are differentiable at $t_0$, the following statements are equivalent.
\begin{enumerate}[{\rm (i)}] 
\item The funciton ${\mathbf a}^{-({\mathbf b }, {\mathbf c})}\colon J\to \A$ is continuous at $t_0$.
\item  The funciton ${\mathbf a}^{-({\mathbf b}, {\mathbf c})}\colon J\to \A$ is differentiable at $t_0$.
\end{enumerate}
\noindent Moreover, in this case,
\begin{align*}
({\mathbf a}^{-({\mathbf b }, {\mathbf c})})'(t_0)=&{\mathbf a}(t_0)^{-({\mathbf b(t_0)}, {\mathbf c(t_0)})}({\mathbf c}^\dag{\mathbf c})'(t_0){\mathbf a}(t_0)^{-({\mathbf b(t_0)}, {\mathbf c(t_0)})}\\
&-(\uno-{\mathbf a}(t_0)^{-({\mathbf b(t_0)}, {\mathbf c(t_0)})}{\mathbf a}(t_0))({\mathbf b}{\mathbf b}^\dag)'(t_0){\mathbf a}(t_0)^{-({\mathbf b(t_0)}, {\mathbf c(t_0)})}\\
& -{\mathbf a}(t_0)^{-({\mathbf b(t_0)}, {\mathbf c(t_0)})}{\mathbf a}'(t_0){\mathbf a}(t_0)^{-({\mathbf b(t_0)}, {\mathbf c(t_0)})}.\\
\end{align*}
\end{cor}
\begin{proof} Apply Theorem \ref{thm6.1} to the case under consideration.
\end{proof}
As an application of Theorem \ref{thm6.1}, a characterization of the differentiability of the Moore-Penrose inverse in the Banach context will be given.

\begin{cor}\label{cor6.2} Let $\A$ be a Banach algebra and consider a function ${\mathbf a}\colon J\to \A^\dag$, where $J\subseteq\mathbb{R}$ is an open set. 
Suppose that there exists $t_0\in J$ such that  the function ${\mathbf a}\colon J\to \A$ is  differentiable at $t_0$. Then, the following statements are equivalent.
\begin{enumerate}[{\rm (i)}] 
\item The function ${\mathbf a}^\dag\colon J\to \A$ is differentiable at $t_0$.
\item The function ${\mathbf a}^\dag\colon J\to \A$ is continuous at  $t_0$ and the functions ${\mathbf a}{\mathbf a}^\dag$, ${\mathbf a}^\dag {\mathbf a}\colon J\to \A$
are differentiable at $t_0$.
\end{enumerate}
Furthermore, in this case, 
\begin{align*}
({\mathbf a}^\dag)'(t_0)=&{\mathbf a}^\dag(t_0)({\mathbf a}{\mathbf a}^\dag)'(t_0){\mathbf a}^\dag(t_0) -(\uno-{\mathbf a}^\dag {\mathbf a}(t_0))({\mathbf a}^\dag {\mathbf a})'(t_0){\mathbf a}^\dag(t_0)
-{\mathbf a}^\dag(t_0){\mathbf a}'(t_0){\mathbf a}^\dag(t_0).\\
\end{align*}
\end{cor}
\begin{proof} Recall that given $x\in\A^\dag$, $x^\dag= x^{-(x^\dag, x^\dag)}$ ({\hspace{-1pt}}\cite[Propostion 6.1]{D})). To conclude the proof apply Theorem \ref{thm6.1} to the function
${\mathbf a}$, ${\mathbf b}$, ${\mathbf c}\colon J\to\A$, where ${\mathbf b}={\mathbf c}={\mathbf a}^\dag$.
\end{proof}

In the frame of $C^*$-algebras, the hypotheses of Corollary \ref{cor6.3} can be lightened.

\begin{cor}\label{cor6.4} Let $\A$ be a $C^*$-algebra and consider a function ${\mathbf a}\colon J\to \A$, where $J\subseteq\mathbb{R}$ is an open set.
Let ${\mathbf b}$, ${\mathbf c}\colon J\to \hat{\A}$ be two functions such that for each $t\in J$, ${\mathbf a}(t)^{-({\mathbf b(t)}, {\mathbf c(t)})}$ exists.
Then, given $t_0\in J$, if the functions ${\mathbf a}$, ${\mathbf b}$, ${\mathbf c}\colon J\to \A$ are differentiable at $t_0$, the  functions
${\mathbf b}^\dag$, ${\mathbf c}^\dag\colon J\to \A$ are continuous at $t_0$, and ${\mathbf a}(t_0)^{-({\mathbf b(t_0)}, {\mathbf c(t_0)})}\neq 0$, the following statements are equivalent.
\begin{enumerate}[{\rm (i)}] 
\item The funciton ${\mathbf a}^{-({\mathbf b }, {\mathbf c})}\colon J\to \A$ is continuous at $t_0$.
\item  The funciton ${\mathbf a}^{-({\mathbf b}, {\mathbf c})}\colon J\to \A$ is differentiable at $t_0$.
\end{enumerate}
\noindent Moreover, in this case,
\begin{align*}
({\mathbf a}^{-({\mathbf b }, {\mathbf c})})'(t_0)=&{\mathbf a}(t_0)^{-({\mathbf b(t_0)}, {\mathbf c(t_0)})}(({\mathbf c}^\dag)'(t_0){\mathbf c}(t_0)+{\mathbf c}^\dag(t_0){\mathbf c}'(t_0)){\mathbf a}(t_0)^{-({\mathbf b(t_0)}, {\mathbf c(t_0)})}\\
&-(\uno-{\mathbf a}(t_0)^{-({\mathbf b(t_0)}, {\mathbf c(t_0)})}{\mathbf a}(t_0))({\mathbf b}'(t_0){\mathbf b}^\dag(t_0)+{\mathbf b}(t_0)({\mathbf b}^\dag)'(t_0)){\mathbf a}(t_0)^{-({\mathbf b(t_0)}, {\mathbf c(t_0)})}\\
& -{\mathbf a}(t_0)^{-({\mathbf b(t_0)}, {\mathbf c(t_0)})}{\mathbf a}'(t_0){\mathbf a}(t_0)^{-({\mathbf b(t_0)}, {\mathbf c(t_0)})}.\\
\end{align*}
\end{cor}
\begin{proof} Note that the condition ${\mathbf a}(t_0)^{-({\mathbf b(t_0)}, {\mathbf c(t_0)})}\neq 0$ implies that ${\mathbf b}(t_0)\neq 0$ and ${\mathbf c}(t_0)\neq 0$ (Definition \ref{def1}). 
Thus, according to \cite[Theorem 2.1]{kol}, the functions ${\mathbf b}^\dag$, ${\mathbf c}^\dag\colon J\to \A$ are differentiable at $t_0$. To conclude the proof,  apply Corollary \ref{cor6.3}.
\end{proof}

As in section 5, the results concerning the differentiability of the inverse along an element can be deduced from Theorem \ref{thm6.1}, Corollary \ref{cor6.2} and Corollary \ref{cor6.4}.
The details are left to the interested reader. 


\section{The outer inverse \OIP}

Although similar arguments to the ones used in sections 5 and 6 will be applied to study the continuity and the differentiability of the outer inverse \OIP, the results of this section can not be derived from the corresponding ones concerning the $(b, c)$-inverse. In fact, in what follows operators between two different Banach spaces will be considered. First the gap between subspaces
will be used to characterize the continity of the \OIP.

\begin{thm}\label{thm7.1} Let $\X$ and $\Y$ be Banach spaces and consider $A\in\L(\X, \Y)$ and two subspaces $\T\subseteq \X$
and $\S\subseteq \Y$ such that $A^{(2)}_{\T, \S}$ exists. Let $(A_n)_{n\in\mathbb{N}}\subset\L(\X, \Y)$ and consider $(\T_n)_{n\in\mathbb{N}}$ and   $(\S_n)_{n\in\mathbb{N}}$
two sequences of subspaces of $\X$ and $\Y$, respectively, such that $(A_n)^{(2)}_{\T_n, \S_n}$ exists, for each $n\in\mathbb{N}$. Suppose that $(A_n)_{n\in\mathbb{N}}$ converges to $A$.
The following statements are equivalent.
\begin{enumerate}[{\rm (i)}] 
\item The sequence $((A_n)^{(2)}_{\T_n, \S_n})_{n\in\mathbb{N}}$ converges to $A^{(2)}_{\T, \S}$.
\item The sequences $((A_n)^{(2)}_{\T_n, \S_n}A_n)_{n\in\mathbb{N}}$ and $(A_n(A_n)^{(2)}_{\T_n, \S_n})_{n\in\mathbb{N}}$ converge to $A^{(2)}_{\T, \S}A$ and $AA^{(2)}_{\T, \S}$, respectively.
\item The sequence $((A_n)^{(2)}_{\T_n, \S_n}A_n)_{n\in\mathbb{N}}$ \rm(\it respectively $(\hat{\delta}(\S_n, \S))_{n\in\mathbb{N}}$\rm) \it    converges to $A^{(2)}_{\T, \S}A$ \rm(\it respectively to 0\rm).\it
\item The sequence $(A_n(A_n)^{(2)}_{\T_n, \S_n})_{n\in\mathbb{N}}$ \rm (\it respectively  $(\hat{\delta}(\T_n, \T))_{n\in\mathbb{N}}$\rm) \it converges to $AA^{(2)}_{\T, \S}$  \rm(\it respectively to 0\rm).\it
\item The sequences $(\hat{\delta}(\T_n, \T))_{n\in\mathbb{N}}$ and $(\hat{\delta}(\S_n, \S))_{n\in\mathbb{N}}$
converge to $0$.\par
\end{enumerate}
\end{thm}
\begin{proof} First define 
\begin{align*}
&P=AA^{(2)}_{\T, \S},&  &P_n=A_n(A_n)^{(2)}_{\T_n, \S_n},& &Q=A^{(2)}_{\T, \S}A,&  
&Q_n=(A_n)^{(2)}_{\T_n, \S_n}A_n&\\
\end{align*} 
($P$, $P_n\in\L(\Y)^\bullet$ and $Q$, $Q_n\in \L(\X)^\bullet$, $n\in\mathbb{N}$). Note that 
\begin{align*}
&\R(Q)=\T,& &\R(Q_n)=\T_n,& &\N(P)=\S,& &\N(P_n)=\S_n.&\\
\end{align*}

It is evident that statement (i) implies statement (ii). 

Now, according to \cite[Lemma 3.3]{kr1}, if $(P_n)_{n\in\mathbb{N}}$ converges to $P$ 
(respectively $(Q_n)_{n\in\mathbb{N}}$ converges to $Q$), then  $(\hat{\delta}(\S_n, \S))_{n\in\mathbb{N}}$ (respectively $(\hat{\delta}(\T_n, \T))_{n\in\mathbb{N}}$)
converges to $0$. Thus, statement (ii) implies statement (iii) (respectively statement (iv)), which in turn implies statement (v).

Suppose that statement (v) holds. If $A^{(2)}_{\T, \S}=0$, then $\T=0$ and $\S=\Y$. In particular,  $(\hat{\delta}(\T_n, 0))_{n\in\mathbb{N}}$ converges to $0$.
However, according to  \cite[Chapter 2, Section 2, Subsection 1]{K}, if $\T_n\neq 0$, $\hat{\delta}(\T_n, 0)=1$. Thus, there exists $n_0\in\mathbb{N}$
such that for each $n\ge n_0$, $\T_n=0$, which implies that  $(A_n)^{(2)}_{\T_n, \S_n}=0$ ($n\ge n_0$). To conclude the proof, assume that $A^{(2)}_{\T, \S}\neq 0$ and
appy \cite[Theorem 3.5]{DX}. 
\end{proof}

In the follwing result the Moore-Penrose inverse will be used to characterize the continuity of the outer inverse \OIP. Although it will not be used in this article, recall
that given a Banach space $\X$ and $P$ and $Q\in\L(\X)^\bullet$ such that $P$ and $Q$ are hermitian, if $\R (P)=\R(Q)$, then $P=Q$ (\hspace{-1pt}\cite[Theorem 2.2]{P}). 
In particular, given a subspace $\M\subseteq \X$, there exists at most one hermitian idempotent $R$ such that $\R(R)=\M$.

\begin{thm}\label{thm7.2} Let $\X$ and $\Y$ be Banach spaces and consider $A\in\L(\X, \Y)$ and two subspaces $\T\subseteq \X$
and $\S\subseteq \Y$ such that $A^{(2)}_{\T, \S}$ exists. Let $(A_n)_{n\in\mathbb{N}}\subset\L(\X, \Y)$ and consider $(\T_n)_{n\in\mathbb{N}}$ and   $(\S_n)_{n\in\mathbb{N}}$
two sequences of subspaces of $\X$ and $\Y$, respectively, such that $(A_n)^{(2)}_{\T_n, \S_n}$ exists, for each $n\in\mathbb{N}$. 
Suppose, in addition, that there exit hermitian idempotents $U\in\L(\X)$ \rm (\it respectively $V\in\L(\Y)$\rm ) \it and $U_n\in\L(\X)$ \rm(\it respectively $V_n\in\L(\Y)$\rm ) \it
such that $\R(U)=\T$ and $\R(U_n)=\T_n$ \rm (\it respectively $\N(V)=\S$ and  $\N(V_n)=\S_n$\rm ), \it  $n\in\mathbb{N}$. If $(A_n)_{n\in\mathbb{N}}$ converges to $A$,
then the following statements are equivalent.
\begin{enumerate}[{\rm (i)}] 
\item The sequence $((A_n)^{(2)}_{\T_n, \S_n})_{n\in\mathbb{N}}$ converges to $A^{(2)}_{\T, \S}$.
\item The sequence  $((A_n)^{(2)}_{\T_n, \S_n}A_n)_{n\in\mathbb{N}}$ converges to $A^{(2)}_{\T, \S}A$  and the sequences  $(V(I-V_n))_{n\in\mathbb{N}}$ and $(V_n(I-V))_{n\in\mathbb{N}}$ converge to 0.
\item The sequence $(A_n(A_n)^{(2)}_{\T_n, \S_n})_{n\in\mathbb{N}}$ converges to $AA^{(2)}_{\T, \S}$ and the sequences $((I-U)U_n)_{n\in\mathbb{N}}$ and $((I-U_n)U)_{n\in\mathbb{N}}$
converge to 0.
\item The sequences $(V(I-V_n))_{n\in\mathbb{N}}$, $(V_n(I-V))_{n\in\mathbb{N}}$, $((I-U)U_n)_{n\in\mathbb{N}}$ and $((I-U_n)U)_{n\in\mathbb{N}}$ converge to 0.
\end{enumerate} 
\end{thm}

\begin{proof} Note that $\hat{\delta}(\T_n, \T)=\hat{\delta}(\R (U_n), \R(U))$. Thus, according to \cite[Lemma 2.2]{V2}, the sequence $(\hat{\delta}(\T_n, \T))_{n\in\mathbb{N}}$
converges to 0 if and only if the sequences $((I-U)U_n)_{n\in\mathbb{N}}$ and $((I-U_n)U)_{n\in\mathbb{N}}$ converge to 0. Similarly, since $\hat{\delta}(\S_n, \S)=\hat{\delta}(\R (I-V_n), \R(I-V))$,
according to \cite[Lemma 2.2]{V2}, the sequence $(\hat{\delta}(\S_n, \S))_{n\in\mathbb{N}}$
converges to 0 if and only if the sequences  $(V(I-V_n))_{n\in\mathbb{N}}$ and $(V_n(I-V))_{n\in\mathbb{N}}$ converge to 0. To conclude the proof, apply Theorem \ref{thm7.1}.
\end{proof}

Next the continuity of the outer inverse \OIP will be studied in the context of Hilbert spaces.

\begin{thm}\label{thm7.3} Let $\H$ and $\K$ be Hilbert spaces and consider $A\in\L(\H, \K)$ and two subspaces $\T\subseteq \H$
and $\S\subseteq \K$ such that $A^{(2)}_{\T, \S}$ exists. Let $(A_n)_{n\in\mathbb{N}}\subset\L(\H, \K)$ and consider $(\T_n)_{n\in\mathbb{N}}$ and   $(\S_n)_{n\in\mathbb{N}}$
two sequences of subspaces of $\H$ and $\K$, respectively, such that $(A_n)^{(2)}_{\T_n, \S_n}$ exists, for each $n\in\mathbb{N}$. 
If $(A_n)_{n\in\mathbb{N}}$ converges to $A$,
then the following statements are equivalent.
\begin{enumerate}[{\rm (i)}] 
\item The sequence $((A_n)^{(2)}_{\T_n, \S_n})_{n\in\mathbb{N}}$ converges to $A^{(2)}_{\T, \S}$.
\item The sequences $((A_n)^{(2)}_{\T_n, \S_n}A_n)_{n\in\mathbb{N}}$ and $(A_n(A_n)^{(2)}_{\T_n, \S_n})_{n\in\mathbb{N}}$ converge to $A^{(2)}_{\T, \S}A$ and $AA^{(2)}_{\T, \S}$, respectively.
\item The sequence  $((A_n)^{(2)}_{\T_n, \S_n}A_n)_{n\in\mathbb{N}}$ \rm (\it respectively $(P_{\S_n^\perp}^\perp)_{n\in\mathbb{N}}$\rm ) \it converges to $A^{(2)}_{\T, \S}A$  \rm (\it respectively to 
$P_{\S^\perp}^\perp$\rm).\it
\item The sequence $(A_n(A_n)^{(2)}_{\T_n, \S_n})_{n\in\mathbb{N}}$ \rm(\it respectively $(P_{\T_n}^\perp)_{n\in\mathbb{N}}$\rm) \it converges to $AA^{(2)}_{\T, \S}$ \rm (\it respectively to $P_{\T}^\perp$\rm).\it
\item The sequences $(P_{\S_n^\perp}^\perp)_{n\in\mathbb{N}}$ and $(P_{\T_n}^\perp)_{n\in\mathbb{N}}$ converge to $P_{\S^\perp}^\perp$ and $P_{\T}^\perp$, respectively.
\end{enumerate} 
\end{thm}
\begin{proof} Note that 
\begin{align*}
&U=P_{\T}^\perp,& &U_n=P_{\T_n}^\perp,& &V=P_{\S^\perp}^\perp,& &V_n=P_{\S_n^\perp}^\perp,&\\
\end{align*}
satisfies the hypoteses of Theorem \ref{thm7.2}. First it will be proved that the sequences $((I-U)U_n)_{n\in\mathbb{N}}$ and $((I-U_n)U)_{n\in\mathbb{N}}$
converge to 0 if and only if the sequence $(P_{\T_n}^\perp)_{n\in\mathbb{N}}$ converges to $P_{\T}^\perp$. 

Since $U$, $U_n\in\L(\H)$ are self-adjoint, the sequences $((I-U)U_n)_{n\in\mathbb{N}}$ and $((I-U_n)U)_{n\in\mathbb{N}}$
converge to 0 if and only if the sequences $(U_n(I-U))_{n\in\mathbb{N}}$ and $(U(I-U_n))_{n\in\mathbb{N}}$
converge to 0. In particular, the sequences $((I-U)U_nU)_{n\in\mathbb{N}}$, $((I-U)U_n(I-U))_{n\in\mathbb{N}}$ and  $(UU_n(I-U))_{n\in\mathbb{N}}$
converge to 0. In addition, since $(U(I-U_n))_{n\in\mathbb{N}}$ converges to 0, the sequence $(UU_nU)_{n\in\mathbb{N}}$ converges to $U$. 
However, since
$$
U_n=UU_nU +UU_n(I-U_n)+(I-U)U_nU+(I-U_n)U_n(I-U),
$$
the sequence $(P_{\T_n}^\perp)_{n\in\mathbb{N}}$ converges to $P_{\T}^\perp$. 

A similar argument proves that the sequences $(V(I-V_n))_{n\in\mathbb{N}}$ and $(V_n(I-V))_{n\in\mathbb{N}}$ converge to 0
if and only if the sequence $(P_{\S_n^\perp}^\perp)_{n\in\mathbb{N}}$ converges to $P_{\S^\perp}^\perp$.

To conlcude the proof, apply Theorem \ref{thm7.1} and Theorem \ref{thm7.2}.
\end{proof}

To study the differentiation of the outer inverse \OIP, it is necessary to present a prelimiery result first.
 
\begin{lem}\label{lem7.4} Let $\X$ and $\Y$ be two Banach spaces and consider $A$, $B\in\L(\X, \Y)$. Let $\T$, $\V\subseteq \X$ and $S$, $\U\subseteq \Y$
be two pairs of subspaces such that $A^{(2)}_{\T, \S}$ and $B^{(2)}_{\V, \U}$ exist and consider  idempotents $P_{\T}$, $P_{\V}\in\L (\X)$ and
$P_{\S}$, $P_{\U}\in\L(\Y)$ such that $\R(P_{\T})=\T$,  $\R(P_{\V})=\V$, $\R(P_{\S})=\S$ and $\R(P_{\U})=\U$. Then
\begin{align*}
B^{(2)}_{\V, \U}-A^{(2)}_{\T, \S}=&B^{(2)}_{\V, \U}(P_{\S}-P_{\U})(I_{\Y}-AA^{(2)}_{\T, \S})+(I_{\X}-B^{(2)}_{\V, \U}B)(P_{\V}-P_{\T})A^{(2)}_{\T, \S}\\
& -B^{(2)}_{\V, \U}(B-A)A^{(2)}_{\T, \S}.\\
\end{align*}
\end{lem}
\begin{proof} Let $P_{\T}$, $P_{\V}\in \L(X)^\bullet$ be such that  $\R(P_{\T})=\T$ and  $\R(P_{\V})=\V$. According to \cite[Lemma 1]{LYZW},
\begin{align*}
&(I_{\X}-B^{(2)}_{\V, \U}B)P_{\V}=0,& &P_{\T}A^{(2)}_{\T, \S}=A^{(2)}_{\T, \S}.&\\
\end{align*}
Then,
\begin{align*}
&B^{(2)}_{\V, \U}BA^{(2)}_{\T, \S}-A^{(2)}_{\T, \S}= -(I_{\X}-B^{(2)}_{\V, \U}B)P_{\T}A^{(2)}_{\T, \S}=  (I_{\X}-B^{(2)}_{\V, \U}B)(P_{\V}-P_{\T})A^{(2)}_{\T, \S}.& \\
\end{align*}

Now consider $P_{\S}$, $P_{\U}\in\L({\Y})^\bullet$  such that $\R(P_{\S})=\S$ and $\R(P_{\U})=\U$. According to \cite[Lemma 1]{LYZW},
\begin{align*}
&(I_{\Y}-P_{\S})(I_{\Y}-AA^{(2)}_{\T, \S})=0,& &B^{(2)}_{\V, \U}=B^{(2)}_{\V, \U}(I_{\Y}-P_{\U}).&
\end{align*}
Consequently,
\begin{align*}
B^{(2)}_{\V, \U}-B^{(2)}_{\V, \U}AA^{(2)}_{\T, \S}&=B^{(2)}_{\V, \U}(I_{\Y}-P_{\U})(I_{\Y}-AA^{(2)}_{\T, \S})=B^{(2)}_{\V, \U}((I_{\Y}-P_{\U})-(I_{\Y}-P_{\S}) )(I_{\Y}-AA^{(2)}_{\T, \S})\\
&=B^{(2)}_{\V, \U}(P_{\S}-P_{\U})(I_{\Y}-AA^{(2)}_{\T, \S}).\\
\end{align*}

Therefore,
\begin{align*}
B^{(2)}_{\V, \U}-A^{(2)}_{\T, \S}&= B^{(2)}_{\V, \U}(P_{\S}-P_{\U})(I_{\Y}-AA^{(2)}_{\T, \S}) +B^{(2)}_{\V, \U}AA^{(2)}_{\T, \S}\\
&\hskip.4truecm +(I_{\X}-B^{(2)}_{\V, \U}B)(P_{\V}-P_{\T})A^{(2)}_{\T, \S}-B^{(2)}_{\V, \U}BA^{(2)}_{\T, \S}\\
&= B^{(2)}_{\V, \U}(P_{\S}-P_{\U})(I_{\Y}-AA^{(2)}_{\T, \S})+(I_{\X}-B^{(2)}_{\V, \U}B)(P_{\V}-P_{\T})A^{(2)}_{\T, \S}\\
&\hskip.4truecm -B^{(2)}_{\V, \U}(B-A)A^{(2)}_{\T, \S}.\\
\end{align*}
\end{proof}

In what follows, the differentiation of the outer inverse \OIP will be studied.

\begin{thm}\label{thm7.5} Let $\X$ and $\Y$ be Banach spaces and consider $J\subseteq \mathbb{R}$ an open set.
Suppose that there exist functions ${\mathbf A}\colon J\to \L(\X, \Y)$, ${\mathbf P}\colon J\to \L(\X)^\bullet$  and ${\mathbf Q}\colon J\to\L(\Y)^\bullet$
such that for each $t\in   J$, $({\mathbf A}(t))^{(2)}_{\R({\mathbf P}(t)),\R({\mathbf Q}(t))}$ exists. If the functions ${\mathbf A}$, ${\mathbf P}$
and ${\mathbf Q}$ are differentiable at $t_0$, then the following statements are equivalent.
\begin{enumerate}[\rm (i)]
\item The function  ${\mathbf A}^{(2)}_{{\mathbf P},{\mathbf Q}}\colon J\to \L(\Y, \X)$ is continuous at $t_0$,
\item the function  ${\mathbf A}^{(2)}_{{\mathbf P},{\mathbf Q}}\colon J\to \L(\Y, \X)$ is differentiable at $t_0$,
\end{enumerate} 
where ${\mathbf A}^{(2)}_{{\mathbf P},{\mathbf Q}}(t)=({\mathbf A}(t))^{(2)}_{\R({\mathbf P}(t)),\R({\mathbf Q}(t))}$.\par
\noindent Furthermore,
\begin{align*}
({\mathbf A}^{(2)}_{{\mathbf P},{\mathbf Q}})'(t_0)=&-{\mathbf A}^{(2)}_{{\mathbf P},{\mathbf Q}}(t_0){\mathbf Q}'(t_0)(I_{\Y}-{\mathbf A}(t_0){\mathbf A}^{(2)}_{{\mathbf P},{\mathbf Q}}(t_0))
+(I_{\X}-{\mathbf A}^{(2)}_{{\mathbf P},{\mathbf Q}}(t_0){\mathbf A}(t_0)){\mathbf P}'(t_0){\mathbf A}^{(2)}_{{\mathbf P},{\mathbf Q}}(t_0)\\
&-{\mathbf A}^{(2)}_{{\mathbf P},{\mathbf Q}}(t_0){\mathbf A}'(t_0){\mathbf A}^{(2)}_{{\mathbf P},{\mathbf Q}}(t_0).\\
\end{align*}
\end{thm}
\begin{proof} It is enough to prove that statement (i) implies statement (ii). This proof can be derived from Lemma \ref{lem7.4}. In fact, according to this result,
\begin{align*}
{\mathbf A}^{(2)}_{{\mathbf P},{\mathbf Q}}(t)-{\mathbf A}^{(2)}_{{\mathbf P},{\mathbf Q}}(t_0)=&-{\mathbf A}^{(2)}_{{\mathbf P},{\mathbf Q}}(t)({\mathbf Q}(t)-
{\mathbf Q}(t_0))(I_{\Y}-{\mathbf A}(t_0){\mathbf A}^{(2)}_{{\mathbf P},{\mathbf Q}}(t_0))\\
&+(I_{\X}-{\mathbf A}^{(2)}_{{\mathbf P},{\mathbf Q}}(t){\mathbf A}(t))({\mathbf P}(t)-{\mathbf P}(t_0)){\mathbf A}^{(2)}_{{\mathbf P},{\mathbf Q}}(t_0)\\
&-{\mathbf A}^{(2)}_{{\mathbf P},{\mathbf Q}}(t)({\mathbf A}(t)-{\mathbf A}(t_0)){\mathbf A}^{(2)}_{{\mathbf P},{\mathbf Q}}(t_0).\\
\end{align*}
\noindent Now divide each term by $t-t_0$ and note that the limit leads to $({\mathbf A}^{(2)}_{{\mathbf P},{\mathbf Q}})'(t_0)$.
\end{proof}

In the Hilbert space operators context, Theorem \ref{thm7.5} can be reformulated as follows.

\begin{cor}\label{cor7.6} Let $\H$ and $\K$ be Hilbert spaces and consider $J\subseteq \mathbb{R}$ an open set.
Suppose that there exist functions ${\mathbf A}\colon J\to \L(\X, \Y)$, ${\mathbf P}^\perp\colon J\to \L(\X)^\bullet$  and ${\mathbf Q}^\perp\colon J\to\L(\Y)^\bullet$
such that for each $t\in   J$, ${\mathbf P}^\perp(t)\in\L(\X)$
 and  ${\mathbf Q}^\perp(t)\in\L(\Y)$ are orthogonal idempotents and $({\mathbf A}(t))^{(2)}_{\R({\mathbf P}^\perp(t)),\R({\mathbf Q}^\perp(t))}$ exists. If the functions ${\mathbf A}$, ${\mathbf P}^\perp$
and ${\mathbf Q}^\perp$ are differentiable at $t_0$, then the following statements are equivalent.
\begin{enumerate}[\rm (i)]
\item The function  ${\mathbf A}^{(2)}_{{\mathbf P}^\perp,{\mathbf Q}^\perp}\colon J\to \L(\Y, \X)$ is continuous at $t_0$,
\item the function  ${\mathbf A}^{(2)}_{{\mathbf P}^\perp,{\mathbf Q}^\perp}\colon J\to \L(\Y, \X)$ is differentiable at $t_0$,
\end{enumerate} 
where ${\mathbf A}^{(2)}_{{\mathbf P}^\perp,{\mathbf Q}^\perp}(t)=({\mathbf A}(t))^{(2)}_{\R({\mathbf P}^\perp(t)),\R({\mathbf Q}^\perp(t))}$.\par
\noindent Furthermore,
\begin{align*}
({\mathbf A}^{(2)}_{{\mathbf P}^\perp,{\mathbf Q}^\perp})'(t_0)=&-{\mathbf A}^{(2)}_{{\mathbf P}^\perp,{\mathbf Q}^\perp}(t_0)({\mathbf Q}^\perp)'(t_0)(I_{\Y}-{\mathbf A}(t_0)
{\mathbf A}^{(2)}_{{\mathbf P}^\perp,{\mathbf Q}^\perp}(t_0))\\
&+(I_{\X}-{\mathbf A}^{(2)}_{{\mathbf P}^\perp,{\mathbf Q}^\perp}(t_0){\mathbf A}(t_0))({\mathbf P}^\perp)'(t_0){\mathbf A}^{(2)}_{{\mathbf P}^\perp,{\mathbf Q}^\perp}(t_0)\\
&-{\mathbf A}^{(2)}_{{\mathbf P}^\perp,{\mathbf Q}^\perp}(t_0){\mathbf A}'(t_0){\mathbf A}^{(2)}_{{\mathbf P}^\perp,{\mathbf Q}^\perp}(t_0).\\
\end{align*}
\end{cor}
\begin{proof} Apply Theorem \ref{thm7.5} to the case under consideration.
\end{proof}


\vskip.3truecm
\noindent Enrico Boasso\par
\noindent E-mail address: enrico\_odisseo@yahoo.it 

\end{document}